\documentclass[11pt]{article}
\usepackage{amssymb,amsmath,amsthm,mathtools}
\usepackage{graphicx}
\usepackage{booktabs}
\usepackage{enumitem}
\usepackage[hidelinks]{hyperref}
\allowdisplaybreaks

\usepackage{geometry}
\numberwithin{equation}{section}
\newtheorem{theorem}{Theorem}[section]
\newtheorem{lemma}[theorem]{Lemma}
\newtheorem{proposition}[theorem]{Proposition}
\newtheorem{corollary}[theorem]{Corollary}

\theoremstyle{remark}
\newtheorem{remark}[theorem]{Remark}
\theoremstyle{plain}

\def\Re{{\mathop{\rm Re}\,}}
\def\Im{{\mathop{\rm Im}\,}}
\newcommand{\R}{\mathbb R}
\newcommand{\C}{\mathbb C}
\newcommand{\N}{\mathcal N}
\newcommand{\E}{\mathcal E}
\newcommand{\Y}{\mathcal Y}
\newcommand{\X}{\mathcal X}
\newcommand{\Z}{\mathcal Z}
\newcommand{\LL}{\mathcal L}
\newcommand{\QQ}{\mathcal Q}
\newcommand{\Cd}{\mathfrak C}
\newcommand{\CP}{\mathrm{CP}}
\newcommand{\norm}[1]{\left\lVert #1\right\rVert}
\newcommand{\ip}[2]{\left(#1,#2\right)_\rho}
\newcommand{\dx}{\,dx}

\begin{document}
\title{Linearize or Discretize First? Error Estimates under Mesh-Uniform
Regularity and Exponential Stability for the KdV Equation with Boundary
Dissipation\thanks{This work is supported by the National Natural Science
Foundation of China (grant No. 12171431).\medskip}}
\author{Kangsheng Liu\thanks{Academy for Advanced Interdisciplinary Studies,
Northeast Normal University, Changchun 130024, Jilin, China; School of
Mathematical Sciences, Zhejiang University, Hangzhou 310058, China
(email: ksliu@zju.edu.cn).}}
\date{}
\maketitle

\begin{center}
\begin{minipage}{6in}
{\bf Abstract.}
We study the Korteweg--de Vries equation on a bounded interval with
boundary dissipation, and its spatial semidiscretization by a weighted
cubic Hermite Galerkin method. A small nonlinear wave can be approximated
by a linear discrete simulation in two orders: linearize first and then
discretize, or discretize first and then linearize. Both orders lead to
the same total error. They divide this
error differently between linearization and discretization. The
difference of the two divisions is what we call the ordering error. We
show that it is quadratic in the size of the initial datum and carries a
positive power of the mesh size, for all time.
It can be used to assess the effect of a coarse-mesh correction for
nonlinear problems.
\par\smallskip
For the continuous problem, the solution map is analytic, and the
semigroup of the linearized equation has Crandall--Pazy regularity with
the exact index $3/2$. For the weighted scheme, the discrete semigroups of
the linearized equation inherit exponential stability, Kato smoothing and
Crandall--Pazy regularity, with bounds uniform in the mesh size. The
nonlinear solution maps are analytic, also uniformly in the mesh size.
\par\smallskip
These properties give error estimates in which the error and the datum
are measured in the same norm: in the energy norm, in $H_0^1$ and in
$H^2$. On every interval $[t_0,\infty)$, $t_0>0$, same-norm convergence
holds uniformly on balls at algebraic rates. The ordering error is of
order $h^{1/6}$ times the square of the energy norm of the datum,
uniformly for all $t\ge0$.
\vskip 3mm

\vskip 4mm

\vskip 3mm
\end{minipage}
\end{center}

\section{Introduction}\label{sec:intro}

A small nonlinear wave is often simulated by a linear discrete model.
Two approximations are then made: the equation is linearized, and it is
discretized. They can be made in either order. Does the order matter?

We study this question for the Korteweg--de Vries (KdV)
equation~\cite{KdV} on an interval of fixed length $L>0$:
\begin{equation}\label{eq:pde}
 u_t+u_{xxx}+uu_x=0,\qquad
 u(0,t)=u(L,t)=u_x(L,t)=0,\qquad u(x,0)=u_0(x).
\end{equation}
We call $u_0$ the initial datum, or simply the datum. Let $v$ solve the
linearized equation with the same datum, and let $u_h$ and $v_h$ be the
Galerkin approximations of $u$ and $v$. Linearizing first gives the path
$u\to v\to v_h$. Discretizing first gives the path $u\to u_h\to v_h$. Both
paths end at $v_h$, so the total error $u-v_h$ is the same. The order
decides how this error is divided between linearization and
discretization. Changing the order moves an amount from one share to the other,
\[
 \Cd_h=(u-v)-(u_h-v_h)=(u-u_h)-(v-v_h).
\]
We call $\Cd_h$ the \emph{ordering error}. It vanishes at $t=0$. Let
$\varepsilon$ be the size of the datum and $h^p$ a mesh factor. Separate
bounds for the four shares give only
$\Cd_h=O(\min\{\varepsilon^2,h^p\varepsilon\})$. Is the ordering error
of the product form $h^p\varepsilon^2$? Does this hold for all time?

For a weighted Galerkin scheme the answer to both questions is yes. Let
$H=L^2((0,L),(L+x)\dx)$, and let $\lambda=3\pi^2/(4L^3)$ be the energy
decay rate described below. For every $0<\nu<\lambda$ there
are $r,C>0$, depending only on $L$, $\nu$ and the mesh regularity, such
that for real data
\begin{equation}\label{eq:headline}
 \sup_{t\ge0}e^{\nu t}\norm{\Cd_h(t)}_H\le Ch^{1/6}\norm{u_0}_H^2
 \qquad\text{if }\norm{u_0}_H\le r\text{ and }0<h\le\min\{1,L\}.
\end{equation}
The bound holds on a whole energy ball and for all time. No higher
norm of the datum enters. The full and linear errors cannot do this. The kernel of
the initial projection contains nonzero smooth functions. Hence $u-u_h$ and $v-v_h$
are not small at $t=0$ uniformly on a ball. Their difference is small
from $t=0$ on. In $H_0^1$ and $H^2$, which control the amplitude and the
slope of the wave, the ordering error is of order $h^{1/6}\varepsilon^2$
and $h^{4/15}\varepsilon^2$ at positive times
(Theorem~\ref{thm:ordering}).

The bound has a direct use. A nonlinear correction computed on a coarse
mesh can be added to a linear solution on a fine mesh. The error of the
sum is the linear error on the fine mesh plus the ordering error on the
coarse mesh (Corollary~\ref{cor:twomesh}).

The starting point is the dissipation. Control and stabilization of KdV
on a bounded interval have been studied
intensively~\cite{Rosier,PMVZ,RZ,RZsurvey,Cerpa}. We omit the transport
term $u_x$, which produces Rosier's critical lengths~\cite{Rosier}, so
that the boundary dissipation stabilizes \eqref{eq:pde} for every $L$.
The $L^2$ norm is dissipated only at the left endpoint. With the
multiplier $L+x$, Rosier's multiplier
method~\cite[Proposition~3.2]{Rosier} shows that the linear flow on the
weighted space $H$ decays like $e^{-\lambda t}$ and has Kato
smoothing~\cite{Kato83} (Proposition~\ref{prop:rosier}).

Three properties enter, each uniformly in the mesh. Exponential stability
makes estimates valid for all time. Analytic dependence on the datum
separates the quadratic part of the solution from the rest. Smoothing in
time gives positive powers of $h$ for data of low regularity. The
delicate one is smoothing, because the KdV semigroup on a bounded
interval is not analytic. Chu, Coron, Shang and Tang~\cite{CCST} proved
that the semigroup generated by $-\partial_x-\partial_x^3$, with the
boundary conditions of \eqref{eq:pde}, is of Gevrey class $s$ for every
$s>3/2$. Resolvent decay of order $|\tau|^{-\beta}$ on the imaginary axis
gives every Gevrey class of order greater than $1/\beta$, by Taylor's
criterion~\cite[Chapter~5]{Taylor89}. Crandall and
Pazy~\cite{CrandallPazy69} showed that such decay makes the semigroup
differentiable, with $\norm{AT(t)}=O(t^{1-2/\beta})$. On Banach spaces,
Eberhardt, El-Mennaoui and Engel~\cite[Theorem~2]{EEE95} obtained
$O(t^{-\gamma})$ for every $\gamma>1/\beta$. Only recently did
Wakaiki~\cite{WakaikiCP} prove the exact characterization: on a Hilbert
space, resolvent decay of order $|\tau|^{-\beta}$ is equivalent to
$\norm{AT(t)}=O(t^{-1/\beta})$ as $t\to0$. We call the second property
Crandall--Pazy regularity of index $1/\beta$. For KdV, $\beta=2/3$ and
the index is $3/2$. A power bound with an explicit index is what
time-weighted error estimates require.

The index $3/2$ is attained and inherited. It is attained by the
linearized semigroup: a wave packet $\chi e^{ikx}$ at time $t\sim k^{-2}$
gives $\norm{AT(t)}\ge ct^{-3/2}$. It is inherited, uniformly in the
mesh, by the cubic Hermite Galerkin scheme in the weighted space $H$.
The scheme cannot do better uniformly
(Theorems~\ref{thm:ctcp} and~\ref{thm:cp}). The weight works because the
Galerkin form satisfies the weighted energy identity exactly on the trial
space. The mesh-uniform exponential decay and Kato smoothing then hold
with the same constants as for the solution to the continuous problem.
The energy identity also gives a resolvent estimate of order
$|\tau|^{-2/3}$, uniform in the mesh.

Whether discretizations inherit the decay of a damped system has been
studied extensively. For the wave equation the answer is negative in
general. Numerical waves of high frequency may travel with vanishing
group velocity~\cite{Trefethen}, and then they escape boundary
observation and damping~\cite{InfanteZuazua,Zuazua05}. Uniform decay can
be restored by numerical viscosity or by other modifications of the
scheme~\cite{TTZ,MunchPazoto,RTT}. Liu and Zheng~\cite{LiuZheng94}
characterized uniform exponential stability of a family of uniformly
bounded semigroups on Hilbert spaces by a uniform resolvent bound on the
imaginary axis. For the
critical generalized KdV equation with localized damping, Pazoto,
Sep\'ulveda and Vera Villagr\'an~\cite{PSV} proved exponential decay for
small data, uniformly in the mesh parameters. We refer to
Tr\'elat~\cite{Trelat} for a survey. Mesh-uniform regularity has a
different history. Parabolic semigroups are analytic in time, and their
discretizations have a natural uniform regularity. Earlier work has
therefore concentrated on them. For non-analytic semigroups we are not
aware of mesh-uniform Crandall--Pazy regularity results.

The Hilbert direct sum of the discrete generators gives a criterion for
mesh-uniform regularity and, at the same time, for uniform stability. A
mesh-uniform resolvent bound is a resolvent
bound for this single generator. By Wakaiki's theorem, a family of
contraction semigroups on Hilbert spaces is uniformly of Crandall--Pazy
class if and only
if it satisfies a uniform resolvent decay estimate
(Proposition~\ref{prop:uniformCP}). This is the regularity counterpart
of the characterization of Liu and Zheng~\cite{LiuZheng94}, which can be
read as the Gearhart--Huang--Pr\"uss theorem~\cite{Gearhart,Huang,Pruss}
for the direct sum. The direct sum of a uniformly regular family again
has Crandall--Pazy regularity. Hence its growth bound equals its spectral
bound. Uniform exponential stability of such a family is therefore
characterized by the spectral bound of the direct sum. Both criteria are
characterizations. Hence this one is equivalent to that of Liu and Zheng.
It replaces their resolvent bound by a spectral bound. For our scheme a
comparison of resolvents with the continuous problem controls this
spectral bound on sufficiently fine meshes, at every rate below
$-s(A)$, the optimal decay rate of the continuous semigroup. Here $s(A)$
is the spectral bound of the generator, and $-s(A)>8\lambda$.

Two more steps complete the argument. Analyticity of the solution maps
is proved by a Catalan majorant for the Duhamel power series in
exponentially weighted trajectory spaces. At the energy level the
nonlinearity loses one derivative. Kato smoothing, integrated over all
times, compensates for this loss. For the ordering error the key is
a cancellation. In its mild form the linear error $v-v_h$, which is not
small uniformly on a ball, drops out. Each remaining term is of second order in the datum and contains a discretization error. Mixed time norms of
type $L^4$ and $L^{4/3}$ control these terms on a whole energy ball.

Along the way we obtain error estimates for the full approximations.
Numerical methods for KdV have a long history. Galerkin,
Petrov--Galerkin, discontinuous Galerkin and finite-difference schemes
were analyzed in~\cite{BDK,DK,Shen,BCKX,HKR,CLR}. Most of these works
concern the periodic problem or the real line. Where they prove error
estimates, these hold on finite time intervals. The estimates require a
higher norm of the solution, or of the datum, than the norm of the
error. Our error bounds hold on the whole half-line and decay
exponentially, with constants independent of time. Except in
Theorem~\ref{thm:higher}, the datum is measured in the norm of the error: in $H$, in
$H_0^1$ or in a compatible $H^2$ space. In time-integrated norms the
errors are of order $h^{1/6}$ in $L^4(L^2)$ for energy data and
$h^{2/3}$ in $L^2(H^1)$ for data in $H_0^1$ that are small only in $H$,
uniformly on balls. On
$[t_0,\infty)$, $t_0>0$, pointwise rates hold with explicit time weights
(Theorems~\ref{thm:convergence}, \ref{thm:samenorm}
and~\ref{thm:higher}).

We are not aware of an earlier study of the ordering error, nor of
mesh-uniform analyticity of solution maps for discretizations of KdV.
Analytic dependence on the data is classical for
KdV~\cite{Bourgain,KPV,BSZ}. On a bounded interval, Bona, Sun and
Zhang~\cite{BSZ} proved it on finite time intervals. Here the Taylor
remainder is uniform for all time. Data, solution and remainder are
measured in the same norm. Without dissipation this form is not to be
expected. On the line, Ifrim, Koch and Tataru~\cite{IKT} obtain
linear-type dispersive bounds for times of order $\varepsilon^{-3}$, for
localized data of size $\varepsilon$. In parabolic problems, exponential
decay and an integrable smoothing kernel allow a contraction in weighted
trajectory spaces~\cite{Henry}. For KdV the derivative loss requires
Kato smoothing at the energy level.

The paper is organized as follows. Section~\ref{sec:continuous} treats
the continuous problem: the analytic solution map, Crandall--Pazy
regularity and exponential decay. Section~\ref{sec:discrete} introduces
the weighted scheme and proves its mesh-uniform properties.
Section~\ref{sec:samenorm} defines the ordering error and proves the
error estimates in the norm of the datum.

Throughout, $x_\wedge=\min\{1,x\}$ for $x\ge0$, for instance $t_\wedge$
and $(t-s)_\wedge$. We write
$\langle\omega\rangle=(1+\omega^2)^{1/2}$. The letter $C$ denotes a
positive constant that may change from line to line. Unless stated
otherwise, it depends only on $L$, on the mesh-regularity constant and
on the decay rates.

\section{The continuous problem}\label{sec:continuous}
\setcounter{equation}{0}

In this section, we treat the continuous problem. We first introduce the
energy space, the adjoint and the fractional scale. We then state that
the solution map is analytic, and we prove Crandall--Pazy regularity of
exact index $3/2$ and exponential decay. The proof of
Theorem~\ref{thm:ctcp}(i) uses the discrete resolvent bound, which is
proved independently in Section~\ref{sec:uniformcp}. The last subsection
proves the analyticity.

\subsection{Energy, adjoint and fractional spaces}\label{sec:setting}
Set $\rho(x)=L+x$ and $H=L^2((0,L),\rho\dx)$, with inner product
$\ip fg=\int_0^Lf\bar g\rho\dx$ and norm $\norm\cdot_H$. We write
$\norm\cdot_2$ for the unweighted $L^2$ norm and $|f|_1=\norm{f'}_2$. A
prime always denotes a spatial derivative. Define
\begin{equation}\label{eq:operator}
 A=-\partial_x^3,\qquad
 D(A)=\{f\in H^3(0,L): f(0)=f(L)=f'(L)=0\},
\end{equation}
and $F(f,g)=-\frac12(fg)'$, $\N(f)=F(f,f)$. Then \eqref{eq:pde} reads
$u_t=Au+\N(u)$. The bilinear map $F$ involves no complex conjugation. Complex
spaces are used in the analyticity arguments. The physical solutions are
real.

For $f\in H^2$ with $f(0)=f(L)=f'(L)=0$ and $\psi\in H_0^1(0,L)$ put
\begin{equation}\label{eq:form}
 b(f,\psi)=-\int_0^Lf''\,\overline{(\rho\psi)'}\dx .
\end{equation}
Integrating by parts once, we obtain $\ip{Af}\psi=-b(f,\psi)$ for $f\in D(A)$.
Put
\[
 a=\frac32,\qquad p=\frac{\sqrt{2L}\,L}{\pi},\qquad
 \lambda=\frac a{p^2}=\frac{3\pi^2}{4L^3}.
\]
Two further integrations by parts yield
\begin{equation}\label{eq:energy}
 \Re b(f,f)=a|f|_1^2+\frac L2|f'(0)|^2,\qquad
 \norm f_H\le p|f|_1 .
\end{equation}
The inequality combines $\rho\le2L$ with Poincar\'e's inequality.

The next proposition is a modification of~\cite[Propositions~3.1 and~3.2]{Rosier}.
There the space is $L^2(0,L)$, the equation contains the transport term,
and the multiplier estimates hold on finite time intervals.

\begin{proposition}[Contraction, decay and Kato smoothing]\label{prop:rosier}
The operator $A$ is invertible, $-A$ is maximal accretive, and $A$
generates a contraction semigroup $T(t)=e^{tA}$ on $H$. For $f\in H$ and
$t\ge0$,
\begin{equation}\label{eq:ctenergy}
 \norm{T(t)}_{H\to H}\le e^{-\lambda t},\qquad
 \norm{T(t)f}_H^2+2a\int_0^t|T(s)f|_1^2\,ds\le\norm f_H^2 .
\end{equation}
\end{proposition}
\begin{proof}
By \eqref{eq:energy}, $-\Re\ip{Af}f=\Re b(f,f)\ge0$ for $f\in D(A)$. For
$s\ge0$ the homogeneous third-order boundary value problem $(s-A)f=0$
therefore has only the zero solution. Hence its three-by-three boundary matrix
is invertible. Variation of constants then solves $(s-A)f=g$ for every
$g\in H$, with $\norm f_{H^3}\le C\norm g_H$. Hence $A$ is invertible,
$-A$ is maximal accretive, and $A$ generates a contraction semigroup.

Let $f\in D(A)$ and $u(t)=T(t)f$. By \eqref{eq:energy},
\[
 \frac d{dt}\norm{u(t)}_H^2=-2\Re b(u,u)=-2a|u|_1^2-L|u'(0)|^2 .
\]
This is the multiplier identity~\cite[(3.16)]{Rosier} with $q=L+x$,
without the transport term. By the Poincar\'e inequality in \eqref{eq:energy} and $\lambda=a/p^2$, we have $2a|u|_1^2\ge2\lambda\norm u_H^2$. Gronwall's inequality then yields the first estimate. Integrating the identity over $(0,t)$, we obtain the second. Both extend
to $f\in H$ by density.
\end{proof}
The second estimate is the Kato smoothing of the linear flow.
Compare~\cite[(3.18)]{Rosier}.

Three integrations by parts identify the $H$-adjoint of $A$:
\begin{equation}\label{eq:adjoint}
\begin{gathered}
 A^*g=\rho^{-1}(\rho g)''',\qquad
 D(A^*)=\{g\in H^3:g(0)=g(L)=g'(0)=0\},\\
 -\Re\ip{A^*g}g=a|g|_1^2+L|g'(L)|^2,\qquad g\in D(A^*).
\end{gathered}
\end{equation}

Define the fractional scale
\begin{equation}\label{eq:scale}
 B=-A,\qquad V_\theta=D(B^{\theta/3}),\qquad
 \norm f_{V_\theta}=\norm{B^{\theta/3}f}_H,\qquad 0\le\theta\le3 .
\end{equation}
In particular $V_0=H$, $V_1=H_0^1(0,L)$,
\[
 V_2=\{f\in H^2:f(0)=f(L)=f'(L)=0\},\qquad V_3=D(A),
\]
and at these integer indices the norms are equivalent to the Sobolev
norms. By the half-plane functional calculus for Hilbert-space contraction semigroups, we have $\norm{B^{is}}\le e^{\pi|s|/2}$ for real
$s$~\cite[Theorem~11.12]{Haase}. By interpolation of fractional powers and of boundary spaces~\cite{Seeley}, we obtain
\begin{equation}\label{eq:interp}
 [V_r,V_s]_\alpha=V_{(1-\alpha)r+\alpha s},\qquad
 0\le r,s\le3,\quad 0\le\alpha\le1 .
\end{equation}
Away from $\theta=1/2,3/2$, $V_\theta$ is the subspace of $H^\theta$
carrying exactly the boundary conditions of order $j<\theta-1/2$. At the
critical indices, $V_\theta$ is given by interpolation. For instance,
$V_{1/2}=H_{00}^{1/2}$. In all cases $V_\theta\hookrightarrow H^\theta$.
We shall often use the moment inequality
\begin{equation}\label{eq:moment}
 \norm f_{V_\theta}\le C\norm f_H^{1-\theta/s}\norm f_{V_s}^{\theta/s},
 \qquad 0\le\theta\le s\le3,\quad s>0 .
\end{equation}
For $1\le\theta\le3$ the nonlinearity satisfies
\begin{equation}\label{eq:product}
 \norm{F(f,g)}_{V_{\theta-1}}
 \le C_\theta\norm f_{V_\theta}\norm g_{V_\theta}.
\end{equation}
At integer indices this follows from the one-dimensional Sobolev product estimates and
the boundary conditions. For $\theta=2$, $(fg)'$ vanishes at both ends.
For $\theta=3$ we also have
$(fg)''(L)=f''(L)g(L)+2f'(L)g'(L)+f(L)g''(L)=0$. The estimates at the intermediate indices, including the critical ones, follow by bilinear complex interpolation.

\subsection{The analytic solution map}\label{sec:analytic}
Let $V$ be a Banach space and $\nu>0$. We define
\begin{gather*}
 \Y_\nu(V)=\Big\{w\in C([0,\infty);V):\
 \norm w_{\Y_\nu(V)}:=\sup_{t\ge0}e^{\nu t}\norm{w(t)}_V<\infty\Big\},\\
 C_\nu^0(V)=\Big\{w\in\Y_\nu(V):\
 \lim_{t\to\infty}e^{\nu t}\norm{w(t)}_V=0\Big\}.
\end{gather*}
The space $C_\nu^0(V)$ is a closed subspace of $\Y_\nu(V)$. For
$\sigma>0$ put
\begin{equation}\label{eq:Enorm}
 \E_\sigma(w)=\sup_{t\ge0}e^{\sigma t}\norm{w(t)}_H+
 \Big(\int_0^\infty e^{2\sigma t}|w(t)|_1^2\,dt\Big)^{1/2}.
\end{equation}
For fixed $\nu$, the energy trajectory space is
\begin{equation}\label{eq:X0}
 \X_0=C_\nu^0(H)\cap L^2(e^{2\nu t}dt;H_0^1),\qquad
 \norm w_{\X_0}=\E_\nu(w).
\end{equation}
For $1\le\theta\le3$ let $\X_\theta=\X_0\cap C_\nu^0(V_\theta)$, with the
norm $\E_\nu(w)+\norm w_{\Y_\nu(V_\theta)}$. For $0<\theta<1$ let
$\X_\theta=[\X_0,\X_1]_\theta$ be the complex interpolation space. We
show in Section~\ref{sec:analyticproof} that
$\X_\theta\hookrightarrow C_\nu^0(V_\theta)$. Thus, for every $\theta$,
the norm of $\X_\theta$ dominates both $\E_\nu$ and
$\norm\cdot_{\Y_\nu(V_\theta)}$. When the weight must be displayed we
write $\X_{0,\sigma}$ for the space \eqref{eq:X0} with $\nu$ replaced by
$\sigma$.

\begin{theorem}[Analytic solution map]\label{thm:continuous}
Fix $0\le\theta\le3$ and $0<\nu<\lambda$. There are $r>0$, $M>0$, both
depending only on $L,\theta,\nu$, and continuous homogeneous polynomials
$P_n:V_\theta\to\X_\theta$ of degree $n$, with $P_1(f)=T(\cdot)f$, such
that
\begin{equation}\label{eq:Pn}
 \norm{P_n(f)}_{\X_\theta}\le M\Big(\frac{\norm f_{V_\theta}}{2r}\Big)^n,
 \qquad f\in V_\theta,\ n\ge1 .
\end{equation}
For complex $f\in V_\theta$ with $\norm f_{V_\theta}<r$ the series
\[
 \mathcal U(f)=\sum_{n\ge1}P_n(f)
\]
converges absolutely in the norm of $\X_\theta$, uniformly on this ball.
Its sum is the unique solution of the mild form of \eqref{eq:pde} in a
ball of $\X_\theta$ whose radius depends only on $L,\theta,\nu$. It is
real for real $f$. For every integer $m\ge1$ there is $C_m$ such that
\begin{equation}\label{eq:taylor}
 \norm{\mathcal U(f)-\sum_{n=1}^mP_n(f)}_{\X_\theta}
 \le C_m\norm f_{V_\theta}^{m+1},\qquad \norm f_{V_\theta}<r .
\end{equation}
In particular, with $\varepsilon=\norm{u_0}_{V_\theta}<r$ and
$u=\mathcal U(u_0)$,
\begin{equation}\label{eq:continuous}
 \norm{u(t)}_{V_\theta}\le C\varepsilon e^{-\nu t},\qquad
 \norm{u(t)-T(t)u_0}_{V_\theta}\le C\varepsilon^2e^{-\nu t},\qquad
 t\ge0 .
\end{equation}
\end{theorem}

Relative to the norm of the datum, the linearization error is
$O(\varepsilon)$ for all time. At $\theta=0$ the series also converges in
the weighted $L^2(H_0^1)$ norm. This form of the statement survives
discretization (Section~\ref{sec:uniformanalytic}).

\subsection{Crandall--Pazy regularity and exponential decay}\label{sec:cp}
Let $S$ be a $C_0$-semigroup with generator $G$ on a Hilbert space. It is
\emph{immediately differentiable} if $S(t)$ maps into $D(G)$ for every
$t>0$. For $0<\beta\le1$ it belongs to the \emph{Crandall--Pazy class}
$\CP_\beta$ if, in addition,
$\limsup_{t\to0+}t^{1/\beta}\norm{GS(t)}<\infty$. We then say that $S$
has \emph{Crandall--Pazy regularity} of index $1/\beta$.

Let $s(G)$ be the spectral bound of $G$ and $\omega_0(G)$ the growth
bound of $S$. An immediately differentiable semigroup is norm continuous
on $(0,\infty)$. Hence, by the spectral mapping theorem~\cite[Section~IV.3]{EngelNagel},
\begin{equation}\label{eq:sdg}
 \omega_0(G)=s(G).
\end{equation}
This holds in particular for every semigroup of class $\CP_\beta$.

\begin{theorem}[Wakaiki~{\cite[Theorem~2.3]{WakaikiCP}}]\label{thm:wakaiki}
Let $G$ generate a $C_0$-semigroup $S$ on a Hilbert space and let
$0<\beta\le1$. Then $S\in\CP_\beta$ if and only if there are
$\Omega,C>0$ such that $i\omega\in\rho(G)$ and
$\norm{R(i\omega,G)}\le C|\omega|^{-\beta}$ for all $|\omega|\ge\Omega$.
\end{theorem}
In~\cite{WakaikiCP} the generator is denoted by $-A$ and the semigroup by $e^{-tA}$. With $G=-A$ the
statement there is the one above. See also~\cite[Theorem~5.1]{Wakaiki} for
general resolvent growth functions.

\begin{theorem}[Crandall--Pazy regularity of exact index]\label{thm:ctcp}
The following hold with constants depending only on $L$.
\begin{enumerate}[label=\rm(\roman*)]
\item $i\R\subset\rho(A)$ and $\norm{R(i\tau,A)}\le C\langle\tau\rangle^{-2/3}$
 for $\tau\in\R$.
\item For $0\le\alpha\le1$ and $t>0$,
 $\norm{B^\alpha T(t)}_{H\to H}\le Ct_\wedge^{-3\alpha/2}e^{-\lambda t}$.
\item There are $c_1,t_0>0$ such that $\norm{AT(t)}_{H\to H}\ge c_1t^{-3/2}$
 for $0<t\le t_0$.
\end{enumerate}
In particular $T\in\CP_{2/3}$, and the index $3/2$ cannot be improved.
\end{theorem}
\begin{proof}
(i) By \eqref{eq:ctenergy}, every $z$ with $\Re z>-\lambda$ lies in
$\rho(A)$, and $\norm{R(i\tau,A)}\le\lambda^{-1}$. The decay bound in (i) follows from
the mesh-uniform bound \eqref{eq:resolvent} on uniform meshes
($\kappa=1$), since $R(i\tau,A_h)\Pi_hf\to R(i\tau,A)f$ for every
$f\in H$. The discrete resolvent bound \eqref{eq:resolvent} is proved in
Section~\ref{sec:uniformcp} without using the present theorem.

To prove the convergence, let $y=R(i\tau,A)f$. Then $y\in D(A)$ and
$\norm y_{H^3}\le C\norm{Ay}_H\le C(1+|\tau|)\norm f_H$. With
the residual $d_h$ of Lemma~\ref{lem:projection},
$(i\tau-A_h)\Pi_hy=\Pi_hf+d_h(y)$, so
$R(i\tau,A_h)\Pi_hf-\Pi_hy=-R(i\tau,A_h)d_h(y)$. For $r,\psi\in X_h$ put
$x=R(-i\tau,A_h^*)\psi$. Then $\ip{R(i\tau,A_h)r}\psi=\ip rx$. The real
part of $\ip{(-i\tau-A_h^*)x}x=\ip\psi x$ and \eqref{eq:energy} yield
$a|x|_1^2\le\norm\psi_H\norm x_H\le\lambda^{-1}\norm\psi_H^2$. Hence
$\norm{R(i\tau,A_h)r}_H\le(a\lambda)^{-1/2}\norm r_{-1,h}$. By
\eqref{eq:residual} with $s=3$,
\[
 \norm{R(i\tau,A_h)\Pi_hf-\Pi_hy}_H
 \le Ch\norm y_{H^3}\le Ch(1+|\tau|)\norm f_H\longrightarrow0 .
\]
Since $\Pi_hy\to y$, the convergence follows. Thus
$\norm{R(i\tau,A)f}_H=\lim_{h\to0}\norm{R(i\tau,A_h)\Pi_hf}_H
\le C\langle\tau\rangle^{-2/3}\norm f_H$, which is (i).

(ii) By (i) and Theorem~\ref{thm:wakaiki}, $\norm{BT(t)}\le Ct^{-3/2}$ on
some interval $(0,t_1]$. For $t_1\le t\le1$ we use
$BT(t)=T(t-t_1)BT(t_1)$. For $t\ge1$ we use $BT(t)=T(t-1)BT(1)$ and
\eqref{eq:ctenergy}. This proves (ii) for $\alpha=1$. The general case follows from the moment inequality $\norm{B^\alpha y}\le C\norm y^{1-\alpha}\norm{By}^\alpha$, applied to $y=T(t)x$.

(iii) Choose $\chi\in C_c^\infty(0,L)$ with $\norm\chi_H=1$ and set
$f_k=\chi e^{ikx}$, for real $k\ge1$. Then $f_k\in D(A)$ and
\[
 r_k:=(A-ik^3)f_k=e^{ikx}(-\chi'''-3ik\chi''+3k^2\chi'),\qquad
 \norm{r_k}_H\le C_0k^2 .
\]
By the contraction property and the variation of constants formula, we obtain
$\norm{T(t)f_k-e^{ik^3t}f_k}_H\le C_0tk^2$. Fix $b_0>0$ with
$C_0b_0\le1/4$ and take $t=b_0/k^2$. Since $A$ commutes with $T(t)$ on $D(A)$, we obtain
\[
 \norm{AT(t)f_k}_H\ge k^3\norm{T(t)f_k}_H-\norm{T(t)r_k}_H
 \ge\tfrac34k^3-C_0k^2\ge\tfrac12k^3
\]
for sufficiently large $k$. Since $k\ge1$ is arbitrary,
the choice $k=(b_0/t)^{1/2}$ gives the lower bound for every sufficiently small
$t$. This proves (iii).
\end{proof}

The proof of \eqref{eq:resolvent} in Section~\ref{sec:uniformcp} uses the
following construction. Let $x\in H_0^1$ with $\norm x_H=1$ and
$k=|x|_1$, so that $pk\ge1$. Fix $\gamma\in(0,L/(8p)]$ and put
$\ell=\gamma/k\le L/8$. We extend $x$ by zero, mollify at scale $\ell$
and cut off in end strips of that scale. This produces
$z\in C_c^\infty(0,L)$ with
\begin{equation}\label{eq:mollify}
 \norm{x-z}_2\le C\ell k,\qquad \norm{z'}_2\le Ck,\qquad
 \norm{z''}_2\le C\ell^{-1}k .
\end{equation}
The cutoff terms are controlled by
$\norm x_{L^2(0,c\ell)}\le c\ell\norm{x'}_{L^2(0,c\ell)}$ and its
analogue at the right end. Each derivative of the cutoff costs one factor
$\ell^{-1}$, which this strip estimate offsets.

Since $T(t)$ commutes with the fractional powers of $B$, it follows from Theorem~\ref{thm:ctcp}(ii) that
\begin{equation}\label{eq:ctsmoothing}
 \norm{T(t)}_{V_s\to V_{s'}}\le Ct_\wedge^{-(s'-s)/2}e^{-\lambda t},
 \qquad 0\le s\le s'\le3,
\end{equation}
and $\norm{T(t)}_{V_s\to V_s}\le e^{-\lambda t}$. One lost spatial
derivative costs $t^{-1/2}$, which is integrable at zero. Let $C_R$ be the
constant in Theorem~\ref{thm:ctcp}(i). A Neumann series around $i\omega$ shows that
\begin{equation}\label{eq:strip}
 z=-\sigma+i\omega\in\rho(A),\qquad
 \norm{R(z,A)}\le2C_R\langle\omega\rangle^{-2/3},\qquad
 0\le\sigma\le(2C_R)^{-1}\langle\omega\rangle^{2/3}.
\end{equation}

The energy rate $\lambda$ is much smaller than the optimal decay rate. Let $s(A)$ be
the spectral bound and $\omega_0(A)$ the growth bound of $T$.

\begin{proposition}[Decay at the spectral bound]\label{prop:spectral}
One has $\omega_0(A)=s(A)$ and
\[
 s(A)\le-\frac{2\pi^3}{L^3}<-8\lambda .
\]
Thus for every $0<\nu<-s(A)$ there is $M_\nu$ with
$\norm{T(t)}\le M_\nu e^{-\nu t}$, $t\ge0$.
\end{proposition}
\begin{proof}
For $\alpha>0$ introduce the equivalent norm
$\norm f_\alpha^2=\int_0^Le^{\alpha x}|f|^2\dx$. Integrating by parts, we obtain, for $f\in D(A)$,
\[
 -\Re(Af,f)_\alpha
 =\frac{3\alpha}{2}\int_0^Le^{\alpha x}|f'|^2\dx
  -\frac{\alpha^3}{2}\norm f_\alpha^2+\frac12|f'(0)|^2 .
\]
For $v=e^{\alpha x/2}f\in H_0^1$,
\[
 \int_0^Le^{\alpha x}|f'|^2\dx
 =\norm{v'}_2^2+\frac{\alpha^2}{4}\norm v_2^2
 \ge\Big(\frac{\pi^2}{L^2}+\frac{\alpha^2}{4}\Big)\norm f_\alpha^2 .
\]
At $\alpha=2\pi/L$ this yields
$-\Re(Af,f)_\alpha\ge(2\pi^3/L^3)\norm f_\alpha^2$. By energy integration in the equivalent norm, we obtain $\omega_0(A)\le-2\pi^3/L^3$. Since
$s(A)\le\omega_0(A)$ and $2\pi^3/L^3=(8\pi/3)\lambda$, the stated
inequality follows.

By Theorem~\ref{thm:ctcp}(ii), $T$ is of class $\CP_{2/3}$. Hence, by \eqref{eq:sdg}, $\omega_0(A)=s(A)$.
\end{proof}

\subsection{Proof of Theorem~\ref{thm:continuous}}\label{sec:analyticproof}
The proof is divided into four steps. We first close the estimates at the energy endpoint
$\theta=0$ in a trajectory space, then treat the scale $1\le\theta\le3$,
then interpolate to $0<\theta<1$, and finally sum the power series.

\emph{The energy endpoint.} The pointwise map $F:L^2\times L^2\to H^{-1}$
is not bounded. For $f_n(x)=n^{1/2}\chi(n(x-x_*))$, with $\chi$ supported
inside the interval, $\norm{f_n}_2$ is bounded whereas
$\norm{(f_n^2)'}_{H^{-1}}\asymp n^{1/2}$. The estimates at the endpoint must therefore be
closed in a trajectory space. Let $V_1^*$ be the dual of $V_1$ with respect to the
weighted pairing. For $y_t=Ay+g$ the weighted energy estimate yields
\begin{equation}\label{eq:forced}
 \E_\nu(y)\le C_{L,\nu}
 \big(\norm{y(0)}_H+\norm{e^{\nu t}g}_{L^2_t(V_1^*)}\big),
 \qquad 0<\nu<\lambda .
\end{equation}
To see this, we multiply the energy identity for $y$, which follows from \eqref{eq:energy}, by $e^{2\nu t}$ and absorb the
source into $(a-\nu p^2)|y|_1^2$, whose coefficient is positive. Smooth
approximation defines the forced solution for such weak sources. Uniform
convergence of the weighted trajectories preserves continuity and decay
at infinity.

For $w,\tilde w\in\X_0$, the weak product bound is
$\norm{F(w,\tilde w)}_{V_1^*}\le C\norm{w\tilde w}_2$. By Agmon's inequality, $\norm{w\tilde w}_2^2\le C\norm w_H|w|_1\norm{\tilde w}_H^2$. If
$M_w=\sup_te^{\nu t}\norm{w(t)}_H$, then
\begin{align*}
 \norm{e^{\nu t}w\tilde w}_{L^2_tL^2_x}^2
 &\le CM_wM_{\tilde w}^2\int_0^\infty e^{-\nu t}|w|_1\,dt
  \le C_{L,\nu}\E_\nu(w)^2\E_\nu(\tilde w)^2 ,
\end{align*}
by the Cauchy--Schwarz inequality with the factor $e^{-2\nu t}$. Thus the
maps
\begin{equation}\label{eq:lowbilinear}
 (\LL f)(t)=T(t)f,\qquad
 \QQ(w,\tilde w)(t)=\int_0^tT(t-s)F(w(s),\tilde w(s))\,ds
\end{equation}
are bounded as $\LL:H\to\X_0$ and $\QQ:\X_0\times\X_0\to\X_0$, over the
complex field as well as the real field. On finite time intervals $F(w,\tilde w)$ also belongs to $L^1_tH$. Hence the weak definition of $\QQ$ agrees with the
usual mild formula.

\emph{The scale $1\le\theta\le3$.} By \eqref{eq:product} and
\eqref{eq:ctsmoothing},
\[
 e^{\nu t}\norm{\QQ(w,\tilde w)(t)}_{V_\theta}
 \le C\int_0^t(t-s)_\wedge^{-1/2}e^{-(\lambda-\nu)(t-s)}e^{-\nu s}\,ds\,
 \norm w_{\Y_\nu(V_\theta)}\norm{\tilde w}_{\Y_\nu(V_\theta)} .
\]
The integral is bounded uniformly in $t$. Hence $\QQ$ is bounded on $C_\nu^0(V_\theta)$.
The decay at infinity is preserved, first for inputs with compact support
in time and then by approximation. Together with the energy endpoint, this shows that
$\LL:V_\theta\to\X_\theta$ and $\QQ:\X_\theta\times\X_\theta\to\X_\theta$
are bounded.

\emph{The scale $0<\theta<1$.} Dirichlet sine truncations show that
$\X_1$ is dense in $\X_0$. The range of the weighted trajectory $t\mapsto e^{\nu t}w(t)$ has compact closure in $H$, which yields uniform convergence of the sine truncations. Dominated convergence
handles the time-integrated $H^1$ norm. By linear and bilinear complex interpolation, we obtain
\[
 \LL:V_\theta\to\X_\theta,\qquad
 \QQ:\X_\theta\times\X_\theta\to\X_\theta .
\]
The evaluation maps $w\mapsto e^{\nu t}w(t)$ have bounds independent of
$t$ at both endpoints. Interpolation and density imply
$\X_\theta\hookrightarrow C_\nu^0(V_\theta)$. We interpolate the linear
and bilinear operators, not the nonlinear solution map.

\emph{The power series.} At each index write $l=\norm\LL$ and
$q=\norm\QQ$. The recursion
\[
 P_1(f)=\LL f,\qquad P_n(f)=\sum_{j+k=n}\QQ(P_j(f),P_k(f))
\]
produces continuous homogeneous polynomials. If $C_n$ are the Catalan
numbers, then
\[
 \norm{P_n(f)}_{\X_\theta}\le C_{n-1}q^{n-1}l^n\norm f_{V_\theta}^n
 \le4^{n-1}q^{n-1}l^n\norm f_{V_\theta}^n .
\]
Put $r=(8ql)^{-1}$. Then $4^{n-1}q^{n-1}l^n=(4q)^{-1}(2r)^{-n}$, which is
\eqref{eq:Pn} with $M=(4q)^{-1}$. The series converges normally on the
ball $\norm f_{V_\theta}<r$. Rearranging the absolutely convergent
double series $\sum_{j,k}\QQ(P_j,P_k)$ shows that the sum solves
$u=\LL f+\QQ(u,u)$. Its norm is at most
$l\norm f/(1-4ql\norm f)\le2l\norm f\le(4q)^{-1}$. On the ball of radius
$(4q)^{-1}$ in $\X_\theta$, the map $w\mapsto\LL f+\QQ(w,w)$ is a
contraction. Indeed, for $w,\tilde w$ in that ball,
$\norm{\QQ(w,w)-\QQ(\tilde w,\tilde w)}\le q(\norm w+\norm{\tilde w})
\norm{w-\tilde w}\le\frac12\norm{w-\tilde w}$. Hence the sum is
the unique fixed point in that ball. Summing the geometric majorant over the degrees $n>m$, we obtain \eqref{eq:taylor}. By conjugation and uniqueness, the values are real for real data. No implicit function theorem is used. This completes the proof of Theorem~\ref{thm:continuous}.
\hfill$\square$

\emph{The quadratic term does not vanish.} For a nonzero real direction
$\phi\in V_\theta$ put
\begin{equation}\label{eq:Wcoeff}
 \Phi=T(\cdot)\phi,\qquad P_2(\phi)=\QQ(\Phi,\Phi).
\end{equation}
For a real scalar $\xi$,
$\mathcal U(\xi\phi)-\xi\Phi=\xi^2P_2(\phi)+O(|\xi|^3)$.
The trajectory $P_2(\phi)$ cannot vanish identically. Otherwise, it follows from its equation that $(\Phi^2)'=0$ in the sense of distributions for positive times.
Smoothing and the zero boundary value imply $\Phi(t)=0$. By strong continuity at zero, we conclude that $\phi=0$. Consequently
$\norm{\mathcal U(\xi\phi)-\xi\Phi}_{\Y_\nu(V_\theta)}\ge
c_\phi\xi^2$ for sufficiently small $|\xi|$.

\section{The weighted scheme and its mesh-uniform properties}\label{sec:discrete}
\setcounter{equation}{0}

In this section, we introduce the weighted scheme and prove its
mesh-uniform properties. These are exponential stability, Kato smoothing,
Crandall--Pazy regularity and analyticity of the solution map. We also
show that, on fine meshes, the decay holds at every rate below $-s(A)$.

\subsection{The scheme and the projection}\label{sec:scheme}
Let $0=x_0<\cdots<x_N=L$ be a mesh with
$h=\max_j(x_j-x_{j-1})\le\min\{1,L\}$ and
$h/\min_j(x_j-x_{j-1})\le\kappa$. Let $X_h\subset V_2$ be the space of
$C^1$ piecewise cubic polynomials that satisfy the three boundary
conditions of \eqref{eq:pde}, and let $I_h$ be the cubic Hermite
interpolant~\cite[Chapter~2]{Ciarlet}. Let $\Pi_h$ be the $H$-orthogonal
projection onto $X_h$. The
weighted Galerkin operator $A_h$ on $X_h$ is defined by
\begin{equation}\label{eq:scheme}
 \ip{A_hw}\psi=-b(w,\psi),\qquad w,\psi\in X_h,
\end{equation}
with the form \eqref{eq:form}. Put $B_h=-A_h$, $T_h(t)=e^{tA_h}$ and
$F_h=\Pi_hF$. The weighted scheme is
\begin{equation}\label{eq:Wscheme}
 u_{h,t}=A_hu_h+F_h(u_h,u_h),\qquad u_h(0)=\Pi_hu_0 .
\end{equation}
All results concern the semidiscretization in space. Time remains continuous. We do
not change the initialization and do not add artificial viscosity. In
this section and below, constants are independent of $h$. They depend
only on $L$, $\kappa$ and the indicated decay rates.

For a functional $g\in V_1^*$ we write $\Pi_hg$ for the element of $X_h$
that represents the restriction of $g$ to $X_h$ in the weighted pairing.
The discrete dual norm is
\[
 \norm r_{-1,h}=\sup_{0\ne\psi\in X_h}\frac{|\ip r\psi|}{|\psi|_1},
\]
and $\norm{\Pi_hg}_{-1,h}\le\norm g_{V_1^*}$.

\begin{lemma}[Projection, residual and graph recovery]\label{lem:projection}
For $w\in V_2$ let $R_hw\in X_h$ be the Ritz projection, defined by
$b(R_hw,\psi)=b(w,\psi)$ for all $\psi\in X_h$. For $w\in D(A)$ let
$d_h(w)=\Pi_hAw-A_h\Pi_hw$ be the residual. For $w\in V_2\cap H^s$,
$s=2,3,4$,
\begin{equation}\label{eq:approx}
 \norm{w-\Pi_hw}_{H^j}\le Ch^{s-j}\norm w_{H^s},\qquad j=0,1,2,
\end{equation}
and $\Pi_h$ is bounded on $V_1$ and on $V_2$. For $w\in D(A)$ one has
$A_hR_hw=\Pi_hAw$. Moreover,
\begin{equation}\label{eq:Ritz}
 \norm{w-R_hw}_{H^2}\le C\norm w_{H^3},\qquad
 \norm{w-R_hw}_{H^2}\le Ch\norm w_{H^4},
\end{equation}
for $w\in V_2\cap H^3$ and $w\in V_2\cap H^4$, respectively. For
$x\in X_h$ and $w\in D(A)$,
\begin{align}
 \norm x_{H^2}&\le C\norm{A_hx}_H,\label{eq:graph}\\
 \norm{A_h\Pi_hw}_H&\le C\norm w_{H^3}.\label{eq:graphproj}
\end{align}
The residual satisfies
\begin{equation}\label{eq:residual}
 \ip{d_h(w)}\psi=\int_0^L(w-\Pi_hw)''\,\overline{(\rho\psi)'}\dx,
 \qquad
 \begin{cases}
 \norm{d_h(w)}_{-1,h}\le Ch^{s-2}\norm w_{H^s},\\
 \norm{d_h(w)}_H\le Ch^{s-3}\norm w_{H^s},
 \end{cases}\quad s=3,4 .
\end{equation}
Finally,
\begin{equation}\label{eq:oneway}
 \norm{B_h^\alpha\Pi_hf}_H\le C\norm{B^\alpha f}_H,\qquad
 0\le\alpha\le1,\ f\in V_{3\alpha}.
\end{equation}
\end{lemma}
\begin{proof}
We use Hermite interpolation~\cite[Theorem~3.1.6]{Ciarlet}, the
best-approximation property of $\Pi_h$ and inverse
estimates~\cite[Section~3.2]{Ciarlet} on $\Pi_hw-I_hw=\Pi_h(w-I_hw)$. This
proves
\eqref{eq:approx} and the $V_2$ stability. For the $V_1$ stability we use the cubic interpolant $Q_hw$ that matches nodal values and sets all nodal
slopes to zero. It satisfies $\norm{(Q_hw)'}_2\le C\norm{w'}_2$ and
$\norm{w-Q_hw}_2\le Ch\norm{w'}_2$. The same decomposition applies.

The Ritz projection exists by coercivity of $b$ on $X_h$. For $w\in D(A)$
we have $\ip{A_hR_hw}\psi=-b(w,\psi)=\ip{Aw}\psi$. Testing with $R_hw-I_hw$ gives $|R_hw-I_hw|_1\le C\norm{(w-I_hw)''}_2$. This bound, one inverse estimate, applied only to this piecewise polynomial, and Hermite interpolation yield \eqref{eq:Ritz}.

For graph recovery we set $w=A^{-1}A_hx$. Then $\norm w_{H^3}\le
C\norm{A_hx}_H$ and $b(w-x,\psi)=0$ for $\psi\in X_h$. We test with $e=x-I_hw$. Coercivity and interpolation give
$\norm{e'}_2\le C\norm{(w-I_hw)''}_2\le Ch\norm{A_hx}_H$. Now \eqref{eq:graph} follows from an inverse estimate and the $H^2$ stability of $I_h$.
Subtracting the weak equations, we obtain the exact residual identity. The two residual bounds follow from approximation and one inverse estimate. The strong residual bound with $s=3$ gives \eqref{eq:graphproj}.

For \eqref{eq:oneway} consider the holomorphic function
$z\mapsto e^{c(z-\alpha)^2}B_h^z\Pi_hB^{-z}$ on the strip
$0\le\Re z\le1$, with $c>0$. On $\Re z=0$ its norm is at most
$e^{c(\alpha^2-y^2)}e^{\pi|y|}$, by the imaginary-power bounds for $B$
and $B_h$, where $y=\Im z$. On $\Re z=1$ an analogous bound holds with the extra
factor $\norm{B_h\Pi_hB^{-1}}\le C$, by \eqref{eq:graphproj}. The factor
$e^{-cy^2}$ controls the vertical growth. Hence \eqref{eq:oneway} follows from the three-lines lemma.
\end{proof}

\subsection{Exponential stability and Kato smoothing}\label{sec:stability}
Equation~\eqref{eq:energy} holds exactly on $X_h$. As in
Proposition~\ref{prop:rosier}, $B_h$ is invertible and maximal
accretive, and
\begin{equation}\label{eq:disenergy}
 \norm{T_h(t)}_{H\to H}\le e^{-\lambda t},\qquad
 |x|_1^2\le a^{-1}\norm{B_hx}_H\norm x_H,\qquad x\in X_h .
\end{equation}
The same proof yields the mesh-uniform Kato smoothing
\begin{equation}\label{eq:disKato}
 \norm{T_h(t)x}_H^2+2a\int_0^t|T_h(s)x|_1^2\,ds\le\norm x_H^2,
 \qquad x\in X_h,\ t\ge0 .
\end{equation}
The adjoint $B_h^*$ on $X_h$ satisfies
$\Re\ip{B_h^*x}x=\Re\ip{B_hx}x\ge a|x|_1^2$. For $y_t=A_hy+g$ in $X_h$
the argument of \eqref{eq:forced} yields, with the same constant,
\begin{equation}\label{eq:disforced}
 \E_\nu(y)\le C_{L,\nu}
 \big(\norm{y(0)}_H+\norm{e^{\nu t}g}_{L^2_t(-1,h)}\big),
 \qquad 0<\nu<\lambda ,
\end{equation}
where $L^2_t(-1,h)$ is the $L^2$ space in time on $(0,\infty)$ with spatial norm $\norm\cdot_{-1,h}$.

For data of small energy the nonlinear scheme keeps the exponential decay at any rate $\nu<\lambda$, with a radius depending on $\nu$.

\begin{proposition}[Nonlinear contraction for small energy]\label{prop:nonlinstab}
Let $0<\nu<\lambda$ and let $u_0\in H$ be real with
$\norm{u_0}_H\le9\pi(1-\nu/\lambda)/(2L)$. Then for every mesh and every
$t\ge0$,
\[
 \norm{u_h(t)}_H\le e^{-\nu t}\norm{\Pi_hu_0}_H,\qquad
 2\nu p^2\int_0^\infty|u_h|_1^2\,dt\le\norm{\Pi_hu_0}_H^2 .
\]
\end{proposition}
\begin{proof}
For real $w\in H_0^1$,
$\ip{F(w,w)}w=-\int_0^Lw^2w'\rho\dx=\frac13\int_0^Lw^3\dx$, since
$\rho'=1$ and $w$ vanishes at both ends. By the bounds $\norm w_\infty\le L^{1/2}|w|_1$, $\norm w_2\le(L/\pi)|w|_1$ and $\norm w_2\le L^{-1/2}\norm w_H$, we obtain
\[
 \tfrac13\Big|\int_0^Lw^3\dx\Big|\le\frac{L}{3\pi}\norm w_H|w|_1^2 .
\]
Since $\Pi_h$ is the $H$-orthogonal projection,
$\ip{\Pi_hF(u_h,u_h)}{u_h}=\ip{F(u_h,u_h)}{u_h}$. Hence, by \eqref{eq:energy}, we have
\[
 \frac12\frac d{dt}\norm{u_h}_H^2
 \le-\Big(a-\frac L{3\pi}\norm{u_h}_H\Big)|u_h|_1^2 .
\]
Since $a=3/2$, the bracket is at least $a\nu/\lambda=\nu p^2$ as long as
$\norm{u_h}_H\le9\pi(1-\nu/\lambda)/(2L)$. Then the norm does not
increase. The bound on $\norm{u_h}_H$ therefore holds for all $t\ge0$. Integrating, we obtain the second estimate. The first follows from the inequality $\norm{u_h}_H\le p|u_h|_1$ and $\lambda=a/p^2$.
\end{proof}
The same computation applies to the continuous solution whenever the
energy identity is available for it, for instance for data in a small
ball of $V_3$. The radius in Proposition~\ref{prop:nonlinstab} is explicit and proportional to $1/L$, in
accordance with the scaling of \eqref{eq:pde}.

Error equations are linearizations along a pair of solutions. We write
$A_\bullet=A$ in the continuous case and $A_\bullet=A_h$ in the discrete
case. In these equations small solutions enter only as coefficients. We treat the
coefficient terms by the forced estimates \eqref{eq:forced} and
\eqref{eq:disforced} and one absorption. For this we shall repeatedly use
\begin{equation}\label{eq:prodenergy}
 \norm{F(f,g)}_H\le C|f|_1|g|_1,\qquad
 \norm{F(f,g)}_{V_1^*}\le C|f|_1\norm g_H,\qquad f,g\in H_0^1 .
\end{equation}
Both follow from integration by parts and the embedding
$H_0^1\hookrightarrow L^\infty$. The second also holds with $f$ and $g$ interchanged. It also
holds for $F_h$ with the norm $\norm\cdot_{-1,h}$.

\subsection{Hilbert direct sums}\label{sec:directsum}
Uniformity in the mesh will be obtained from statements about a single
generator. Let $S_n$, $n\in\mathbb N$, be contraction semigroups with
generators $G_n$ on Hilbert spaces $\mathcal H_n$, and let
$\mathfrak H=\bigoplus_n\mathcal H_n$ be their $\ell^2$ direct sum.
Let $\iota_n:\mathcal H_n\to\mathfrak H$ be the inclusion as the $n$-th
summand and $\pi_n:\mathfrak H\to\mathcal H_n$ the coordinate map, so that
$\pi_n\iota_n=I$ and $\iota_n\pi_n$ is the orthogonal projection onto the
$n$-th summand.

\begin{lemma}[Direct sums]\label{lem:directsum}
The operators $\mathfrak S(t)=\bigoplus_nS_n(t)$ form a contraction
$C_0$-semigroup on $\mathfrak H$. Its generator is
\begin{equation}\label{eq:diagonal}
 \mathfrak G(x_n)_n=(G_nx_n)_n,\qquad
 D(\mathfrak G)=\Big\{(x_n)\in\mathfrak H:\ x_n\in D(G_n),\
 \sum_n\norm{G_nx_n}^2<\infty\Big\}.
\end{equation}
Moreover, for $z\in\C$ and $t>0$:
\begin{enumerate}[label=\rm(\roman*)]
\item $z\in\rho(\mathfrak G)$ if and only if $z\in\bigcap_n\rho(G_n)$ and
 $\sup_n\norm{R(z,G_n)}<\infty$. In this case
 $\norm{R(z,\mathfrak G)}=\sup_n\norm{R(z,G_n)}$.
\item $\mathfrak S(t)\mathfrak H\subset D(\mathfrak G)$ if and only if
 $S_n(t)\mathcal H_n\subset D(G_n)$ for all $n$ and
 $\sup_n\norm{G_nS_n(t)}<\infty$. In this case
 $\norm{\mathfrak G\mathfrak S(t)}=\sup_n\norm{G_nS_n(t)}$.
\end{enumerate}
\end{lemma}
\begin{proof}
\emph{Step 1: the semigroup.} Each $S_n(t)$ is a contraction. Hence
$\mathfrak S(t)$ is a well-defined contraction. The semigroup law holds
componentwise. For $x=(x_n)\in\mathfrak H$,
\[
 \norm{\mathfrak S(t)x-x}^2=\sum_n\norm{S_n(t)x_n-x_n}^2 .
\]
Each term tends to zero as $t\to0$ and is bounded by $4\norm{x_n}^2$.
Strong continuity follows by dominated convergence. By definition,
$\pi_n\mathfrak S(t)=S_n(t)\pi_n$ and $\mathfrak S(t)\iota_n=\iota_nS_n(t)$.

\emph{Step 2: the generator.} Let $\mathfrak G$ be the generator and let
$\mathfrak G_d$ be the diagonal operator on the set $D_d$ displayed in
\eqref{eq:diagonal}. Fix $z$ with $\Re z>0$. The resolvent is the Laplace
transform $R(z,\mathfrak G)y=\int_0^\infty e^{-zt}\mathfrak S(t)y\,dt$.
Applying the bounded map $\pi_n$ under the integral, we obtain
$\pi_nR(z,\mathfrak G)y=R(z,G_n)\pi_ny$. Let $x\in D(\mathfrak G)$ and
$y=(z-\mathfrak G)x$. Then $x_n=R(z,G_n)y_n\in D(G_n)$ and
$G_nx_n=zx_n-y_n=\pi_n\mathfrak Gx$. Hence
$\sum_n\norm{G_nx_n}^2=\norm{\mathfrak Gx}^2<\infty$. Thus
$D(\mathfrak G)\subset D_d$ and $\mathfrak G=\mathfrak G_d$ on
$D(\mathfrak G)$. Conversely, $z-\mathfrak G$ maps $D(\mathfrak G)$ onto
$\mathfrak H$. On $D_d$ the operator $z-\mathfrak G_d$ is injective, because each
$z-G_n$ is injective. If $x\in D_d$, choose $\tilde x\in D(\mathfrak G)$
with $(z-\mathfrak G)\tilde x=(z-\mathfrak G_d)x$. Then
$(z-\mathfrak G_d)(x-\tilde x)=0$. Hence $x=\tilde x\in D(\mathfrak G)$,
and $\mathfrak G=\mathfrak G_d$.

\emph{Step 3: part (i).} By Step~2, $\iota_n$ maps $D(G_n)$ into
$D(\mathfrak G)$ with $\mathfrak G\iota_n=\iota_nG_n$. Moreover, $\iota_n\pi_n$
maps $D(\mathfrak G)$ into itself and commutes with $\mathfrak G$ there.
Let $z\in\rho(\mathfrak G)$. Then $R(z,\mathfrak G)$ commutes with
$\iota_n\pi_n$. For $y_n\in\mathcal H_n$ the element
$x=R(z,\mathfrak G)\iota_ny_n$ satisfies $x=\iota_n\pi_nx$. Hence
$x=\iota_nx_n$ with $x_n\in D(G_n)$ and $(z-G_n)x_n=y_n$. If
$(z-G_n)x_n=0$, then $(z-\mathfrak G)\iota_nx_n=0$ and $x_n=0$. Hence
$z\in\rho(G_n)$ and $R(z,G_n)=\pi_nR(z,\mathfrak G)\iota_n$, whose norm
is at most $\norm{R(z,\mathfrak G)}$. Conversely, let
$z\in\bigcap_n\rho(G_n)$ with $\norm{R(z,G_n)}\le M$ for all $n$. The
diagonal operator $\mathfrak K=\bigoplus_nR(z,G_n)$ has norm
$\sup_n\norm{R(z,G_n)}\le M$. It maps $\mathfrak H$ into $D(\mathfrak G)$,
because $G_nR(z,G_n)=zR(z,G_n)-I$. It is a two-sided inverse of
$z-\mathfrak G$. The norm identity follows from the two inequalities.

\emph{Step 4: part (ii).} If $\mathfrak S(t)\mathfrak H\subset
D(\mathfrak G)$, then $\mathfrak G\mathfrak S(t)$ is closed and
everywhere defined, hence bounded. Applied to $\iota_nx_n$, it yields
$\iota_nG_nS_n(t)x_n$. Hence $\norm{G_nS_n(t)}\le\norm{\mathfrak G\mathfrak
S(t)}$. Conversely, if the blocks are uniformly bounded, then
$\sum_n\norm{G_nS_n(t)x_n}^2\le\sup_n\norm{G_nS_n(t)}^2\norm x^2$. Thus
$\mathfrak S(t)x\in D(\mathfrak G)$ by \eqref{eq:diagonal}, and the norm
identity follows. The proof is complete.
\end{proof}

\begin{proposition}[Uniform Crandall--Pazy regularity]\label{prop:uniformCP}
Let $\{S_j\}_{j\in\mathfrak J}$ be a family of contraction semigroups with
generators $G_j$ on Hilbert spaces, and let $0<\beta\le1$. The following
are equivalent.
\begin{enumerate}[label=\rm(\alph*)]
\item Every $S_j$ is immediately differentiable and
 $\sup_{j\in\mathfrak J}\sup_{0<t\le1}t^{1/\beta}\norm{G_jS_j(t)}<\infty$.
\item There are $\Omega,C>0$ such that $i\omega\in\rho(G_j)$ and
 $\norm{R(i\omega,G_j)}\le C|\omega|^{-\beta}$ for all $j\in\mathfrak J$
 and all $|\omega|\ge\Omega$.
\end{enumerate}
\end{proposition}
\begin{proof}
(b)$\Rightarrow$(a). By Theorem~\ref{thm:wakaiki} every $S_j$ is in
$\CP_\beta$. Suppose that the supremum in (a) is infinite. Then there are
$j_n\in\mathfrak J$ and $t_n\in(0,1]$ with
$t_n^{1/\beta}\norm{G_{j_n}S_{j_n}(t_n)}\ge n$. Let $\mathfrak S$ be the
direct sum of $S_{j_1},S_{j_2},\dots$. By (b) and
Lemma~\ref{lem:directsum}(i),
$\norm{R(i\omega,\mathfrak G)}\le C|\omega|^{-\beta}$ for
$|\omega|\ge\Omega$. By Theorem~\ref{thm:wakaiki}, $\mathfrak S\in\CP_\beta$: there are $t_*\in(0,1]$ and $C'$ with
$\norm{\mathfrak G\mathfrak S(t)}\le C't^{-1/\beta}$ for $0<t\le t_*$. For
$t_*\le t\le1$, $\mathfrak G\mathfrak S(t)=\mathfrak S(t-t_*)\mathfrak
G\mathfrak S(t_*)$ yields
$\norm{\mathfrak G\mathfrak S(t)}\le C't_*^{-1/\beta}\le
C't_*^{-1/\beta}t^{-1/\beta}$. By Lemma~\ref{lem:directsum}(ii), the
same bound holds for every $G_{j_n}S_{j_n}(t)$, a contradiction for large
$n$.

(a)$\Rightarrow$(b). Suppose (b) fails. Then there are $j_n$ and
$\omega_n$ with $|\omega_n|\ge n$ such that either
$i\omega_n\notin\rho(G_{j_n})$ or
$|\omega_n|^\beta\norm{R(i\omega_n,G_{j_n})}\ge n$. By (a) and
Lemma~\ref{lem:directsum}(ii), the direct sum of the $S_{j_n}$ belongs to
$\CP_\beta$. By Theorem~\ref{thm:wakaiki} there are $\Omega',C'$ with
$\norm{R(i\omega,\mathfrak G)}\le C'|\omega|^{-\beta}$ for
$|\omega|\ge\Omega'$. Lemma~\ref{lem:directsum}(i) transfers this bound to
every $G_{j_n}$, which contradicts the choice of $(j_n,\omega_n)$ for
$n>\max\{\Omega',C'\}$.
\end{proof}
Only countable direct sums are used. The index set may be uncountable,
for instance the set of all admissible meshes. The analogous
characterization of uniform exponential stability is due to Liu and
Zheng~\cite{LiuZheng94,LiuZheng99}. It is a uniform version of the
Gearhart--Huang--Pr\"uss theorem~\cite{Gearhart,Huang,Pruss}. In
Section~\ref{sec:spectral} the
direct sum of the discrete semigroups is again of class $\CP_{2/3}$. Hence its growth bound equals its
spectral bound.

\subsection{Uniform resolvent decay and Crandall--Pazy regularity}
\label{sec:uniformcp}
In this subsection, we prove the mesh-uniform resolvent bound and the
Crandall--Pazy regularity of the weighted scheme. The main result reads as
follows.

\begin{theorem}[Uniform resolvent decay and Crandall--Pazy regularity]\label{thm:cp}
For all meshes of Section~\ref{sec:scheme}, with constants independent of
$h$,
\begin{align}
 \norm{(i\tau-A_h)^{-1}}_{H\to H}&\le C\langle\tau\rangle^{-2/3},
 \qquad\tau\in\R,\label{eq:resolvent}\\
 \norm{B_h^\alpha T_h(t)}_{H\to H}&\le Ct_\wedge^{-3\alpha/2}e^{-\lambda t},
 \qquad 0\le\alpha\le1,\ t>0,\label{eq:cp}\\
 |T_h(t)x|_1&\le Ct_\wedge^{-3/4}e^{-\lambda t}\norm x_H,
 \qquad x\in X_h,\label{eq:cpH1}\\
 \norm{T_h(t)\Pi_hg}_H&\le Ct_\wedge^{-3/4}e^{-\lambda t}\norm g_{V_1^*},
 \qquad g\in V_1^*.\label{eq:weaksmoothing}
\end{align}
The continuous operator satisfies the same estimates, with $\Pi_h$
omitted. The exponent $3/2$ in \eqref{eq:cp} with $\alpha=1$ cannot be
improved uniformly in $h$.
\end{theorem}
\begin{proof}
\emph{Resolvent.} Let $x\in X_h$, $\norm x_H=1$, $f=(i\tau-A_h)x$,
$r_f=\norm f_H$ and $k=|x|_1$. Equation~\eqref{eq:energy} holds on
$X_h$. Taking the real part of $\ip fx$, we obtain $ak^2\le r_f$.
Together with $1\le pk$, this yields $r_f\ge\lambda$. Fix a small
$\gamma\in(0,L/(8p)]$ and put $\ell=\gamma/k\le L/8$. If $h\ge\ell$, then, by inverse estimates, $\norm{x''}_2\le Ch^{-1}k\le C\gamma^{-1}k^2$. Hence
$|b(x,x)|\le C_\gamma k^3$. Testing with $x$ then yields
$|\tau|\le r_f+C_\gamma k^3\le Cr_f^{3/2}$.

Suppose $h<\ell$. Let $z\in C_c^\infty(0,L)$ be the function of
Section~\ref{sec:cp}, which satisfies \eqref{eq:mollify}, and take
$\psi=I_hz\in X_h$. On each element $(I_hz)''$ is the $L^2$
projection of $z''$ onto linear polynomials. Indeed, the interpolation error and
its first derivative vanish at both element ends. Hence its second
derivative is orthogonal to all linear functions. Therefore
\[
 \norm{x-\psi}_H\le C(\ell k+h^2\ell^{-1}k)\le2C\gamma,
 \qquad \norm{\psi''}_2\le C\ell^{-1}k,
 \qquad \norm{\psi'}_2\le Ck .
\]
Choose $\gamma$ so that the first bound is at most $1/2$. Then
$|\ip x\psi|\ge1/2$. Since $\psi(0)=\psi'(0)=0$ and $x'(L)=0$, one more
integration by parts yields
\[
 b(x,\psi)=\int_0^Lx'\,\overline{(\rho\psi)''}\dx,
 \qquad |b(x,\psi)|\le C_\gamma k^3 .
\]
Testing $f$ against $\psi$, we obtain
$\frac12|\tau|\le r_f\norm\psi_H+|b(x,\psi)|\le C(r_f+r_f^{3/2})
\le Cr_f^{3/2}$, where the last step uses $r_f\ge\lambda$. Hence
$r_f\ge c\max\{\lambda,|\tau|^{2/3}\}$, which proves \eqref{eq:resolvent}.

\emph{Crandall--Pazy regularity.} By \eqref{eq:resolvent} the family
$\{T_h\}$ satisfies condition (b) of Proposition~\ref{prop:uniformCP}
with $\beta=2/3$. Hence
$\sup_h\sup_{0<t\le1}t^{3/2}\norm{B_hT_h(t)}<\infty$. For $t\ge1$ we write $B_hT_h(t)=T_h(t-1)B_hT_h(1)$ and use \eqref{eq:disenergy}. This proves
\eqref{eq:cp} for $\alpha=1$. The moment inequality for maximal accretive
operators has constants independent of $h$, which yields the case $0<\alpha<1$.
Applying \eqref{eq:disenergy} to $T_h(t)x$, we obtain \eqref{eq:cpH1}.

\emph{Weak sources.} On $X_h$, $B_h$ commutes with $T_h$, so
$\norm{B_h^*T_h(t)^*}=\norm{B_hT_h(t)}$. Combining this with $a|y|_1^2\le\Re\ip{B_h^*y}y$, \eqref{eq:cp} and the contraction property, we obtain
$|T_h(t)^*\psi|_1\le Ct_\wedge^{-3/4}e^{-\lambda t}\norm\psi_H$ for
$\psi\in X_h$. Since
$\ip{T_h(t)\Pi_hg}\psi=\langle g,T_h(t)^*\psi\rangle$, \eqref{eq:weaksmoothing} follows by duality. For the continuous operator we use the adjoint
semigroup and \eqref{eq:adjoint}. The last assertion of the theorem is proved below.
\end{proof}

In particular the family $\{T_h\}$ is uniformly of class $\CP_{2/3}$.
We now prove the last assertion of Theorem~\ref{thm:cp}. By
Theorem~\ref{thm:ctcp}(iii) the index $3/2$ is already optimal for the
continuous semigroup. Suppose that $\norm{B_hT_h(t)}\le Ct^{-\gamma}$ for $0<t\le1$ and
all meshes, for some $\gamma<3/2$. Fix $f\in H$ and $0<t\le1$, and put
$w_h=B_hT_h(t)\Pi_hf$. Then $B_h^{-1}\Pi_hw_h=T_h(t)\Pi_hf$, which tends
to $T(t)f$ as $h\to0$ by Proposition~\ref{thm:positive}.
By Lemma~\ref{lem:resdefect} at $z=0$ and \eqref{eq:approx}, we have
$\norm{B_h^{-1}\Pi_h-B^{-1}}\le Ch$. The family $(w_h)$ is bounded in $H$. Therefore every weak
limit point $w$ of $w_h$ satisfies $B^{-1}w=T(t)f$, that is,
$w=BT(t)f$. Hence $\norm{AT(t)}\le Ct^{-\gamma}$ for $0<t\le1$, which
contradicts Theorem~\ref{thm:ctcp}(iii). The last assertion is not used
elsewhere. Its proof uses Proposition~\ref{thm:positive} and
Lemma~\ref{lem:resdefect}, which do not depend on it.

\subsection{Analyticity of the discrete solution map}\label{sec:uniformanalytic}
In this subsection, we show that the analyticity of
Theorem~\ref{thm:continuous} at $\theta=0$ carries over to the discrete
solution maps, with constants independent of $h$.

\begin{proposition}[Mesh-uniform analyticity]\label{prop:uniformanalytic}
Fix $0<\nu<\lambda$. There is $r>0$, depending only on $L$ and $\nu$,
such that for every mesh the discrete solution map $f\mapsto u_h$,
$u_h(0)=f\in X_h$, extends to a holomorphic map from the complex ball
$\{f\in X_h:\norm f_H<r\}$ into $\X_0$. There are homogeneous
polynomials $P_{n,h}$ of degree $n$ with
\[
 u_h=\sum_{n\ge1}P_{n,h}(f),\qquad P_{1,h}(f)=T_h(\cdot)f,\qquad
 \norm{P_{n,h}(f)}_{\X_0}\le4^{n-1}q^{n-1}l^n\norm f_H^n ,
\]
where $l$ and $q$ are the bounds of $\LL$ and $\QQ$ in
\eqref{eq:lowbilinear}. The series converges absolutely in $\X_0$, and
the Taylor remainder estimate \eqref{eq:taylor} at $\theta=0$ holds for
$u_h$ with constants independent of $h$.
\end{proposition}
\begin{proof}
Equation~\eqref{eq:energy} holds exactly on $X_h$. The weak product bound
and Agmon's inequality use only $H_0^1$, and
$\norm{\Pi_hg}_{-1,h}\le\norm g_{V_1^*}$. Hence \eqref{eq:disforced} and
the bilinear bound \eqref{eq:lowbilinear} hold for $T_h$ and $F_h$ with
the same constants. The proof of Theorem~\ref{thm:continuous} then applies
verbatim. Small discrete solutions are unique, so the sum of the series is
$u_h$.
\end{proof}
Analyticity uniform in $h$ is proved here only at the energy level. On
fine meshes the weight can be raised to any $\nu<-s(A)$
(Proposition~\ref{prop:uniformspectraldecay}).

\subsection{Resolvent comparison and decay below the spectral bound}
\label{sec:spectral}
In this subsection, we compare the discrete resolvents with the continuous
one. The main result is Proposition~\ref{prop:uniformspectraldecay}. It
shows that, on fine meshes, the discrete semigroups decay at every rate below
$-s(A)$. Its proof uses Lemma~\ref{lem:resdefect}. Lemma~\ref{lem:CPcontour}
provides a contour representation of the semigroups with common bounds,
which is used in Section~\ref{sec:samenorm}.

\begin{lemma}[Uniform resolvent defect]\label{lem:resdefect}
Let $C_R$ be a constant such that, by Theorem~\ref{thm:cp}, $\norm{R(i\omega,A)}$ and $\norm{R(i\omega,A_h)}$ are at most
$C_R\langle\omega\rangle^{-2/3}$, and let
$c_\Gamma=\min\{\frac14,(2C_R)^{-1}\}$. For $z=-\sigma+i\omega$ with
$0\le\sigma\le c_\Gamma\langle\omega\rangle^{2/3}$ and $0\le\theta\le3$,
\[
 \norm{\{R(z,A_h)\Pi_h-\Pi_hR(z,A)\}B^{-\theta/3}}_{H\to H}
 \le Ch\langle\omega\rangle^{-\theta/3}.
\]
\end{lemma}
\begin{proof}
Both resolvents exist with norm at most $2C_R\langle\omega\rangle^{-2/3}$,
by a Neumann series around $i\omega$, as in \eqref{eq:strip}. Let
$w=R(z,A)f$. Since $(z-A_h)\Pi_hw=\Pi_hf+d_h(w)$,
\[
 y:=R(z,A_h)\Pi_hf-\Pi_hR(z,A)f=-R(z,A_h)d_h(w).
\]
For $\psi\in X_h$ put $x=R(\bar z,A_h^*)\psi$, where $A_h^*$ is the
$H$-adjoint of $A_h$ on $X_h$. Then $\ip y\psi=-\ip{d_h(w)}x$, so
$|\ip y\psi|\le\norm{d_h(w)}_{-1,h}|x|_1$. Taking the real part of $\ip{(\bar z-A_h^*)x}x=\ip\psi x$ and using \eqref{eq:energy}, we obtain
\[
 a|x|_1^2\le\norm\psi_H\norm x_H+\sigma\norm x_H^2
 \le C\langle\omega\rangle^{-2/3}\norm\psi_H^2 .
\]
By \eqref{eq:residual} with $s=3$,
$\norm{d_h(w)}_{-1,h}\le Ch\norm w_{H^3}\le Ch\norm{AR(z,A)f}_H$.
Finally $\norm{AR(z,A)}\le C\langle\omega\rangle^{1/3}$ and
$\norm{AR(z,A)B^{-1}}=\norm{R(z,A)}\le C\langle\omega\rangle^{-2/3}$.
By complex interpolation on the scale $V_\theta$, we obtain
$\norm{AR(z,A)B^{-\theta/3}}\le C\langle\omega\rangle^{1/3-\theta/3}$.
The three bounds together give the claim.
\end{proof}

\begin{lemma}[A common contour with absolutely convergent integrals]\label{lem:CPcontour}
There is $c_\Gamma'\in(0,c_\Gamma]$, depending only on $L,\kappa$, such
that the upward-oriented curve
$\Gamma:\ z(\omega)=-c_\Gamma'\langle\omega\rangle^{2/3}+i\omega$
lies in the resolvent sets of $G=A$ and of every $G=A_h$, and
\[
 \norm{R(z(\omega),G)}\le C\langle\omega\rangle^{-2/3},\qquad
 e^{tG}=\frac1{2\pi i}\int_\Gamma e^{zt}R(z,G)\,dz,\quad t>0 .
\]
The integral converges in operator norm. For $A_h$ the formula holds on
$X_h$ and may then be composed with $\Pi_h$.
\end{lemma}
\begin{proof}
The strip bound of Lemma~\ref{lem:resdefect} covers $\Gamma$ at large
$|\omega|$. At bounded frequencies we choose the leftward displacement
smaller than $\lambda/2$ and use the contraction property. Thus the region between
$\Gamma$ and any vertical line $\Re z=b>0$ is free of spectrum, and the
resolvent decays there as $|\Im z|\to\infty$. We decrease $c_\Gamma'$ further so
that $|z'(\omega)|\le2$. The integral defines a norm-continuous
family of operators $S_G(t)$, $t>0$, and
\[
 \norm{S_G(t)}\le C\int_{\R}e^{-c_\Gamma't\langle\omega\rangle^{2/3}}
 \langle\omega\rangle^{-2/3}\,d\omega\le C(1+t^{-1/2})e^{-c't}.
\]
In particular its Laplace transform exists for real $s_0>0$. Fubini's
theorem is justified by $\norm{R(z,G)}/(s_0-\Re z)=O(|\omega|^{-4/3})$.
After the time integration the integrand is $R(z,G)/(s_0-z)$, of order
$|\omega|^{-5/3}$ on $\Gamma$.

To identify $S_G(t)$, we multiply by $G^{-1}$, whose norm is at most
$\lambda^{-1}$, and deform the integral of $R(z,G)G^{-1}/(s_0-z)$ to a
vertical line $\Re z=b>s_0$. The integrals over the horizontal joining segments at heights
$\pm R_0$ tend to zero as $R_0\to\infty$. Indeed, the segments have length $O(R_0^{2/3})$, while the
integrand is $O(R_0^{-5/3})$. The only enclosed pole is at $z=s_0$.
Finally we use $R(z,G)G^{-1}=z^{-1}(R(z,G)+G^{-1})$ and
$\norm{R(b+i\omega,G)}\le(b+\lambda)^{-1}$. The regularized integral on
$\Re z=b$ is $O(b^{-1})$ as $b\to\infty$, since
$\int_{\R}|b+i\omega|^{-1}|b-s_0+i\omega|^{-1}d\omega=O(b^{-1})$.
By Cauchy's formula, we now obtain
$\int_0^\infty e^{-s_0t}S_G(t)G^{-1}dt=R(s_0,G)G^{-1}$. By uniqueness of the
Laplace transform, $S_G(t)G^{-1}=e^{tG}G^{-1}$ for $t>0$. The
range of $G^{-1}$ is dense, so the formula follows. All the estimates
just used have common constants. The proof is complete.
\end{proof}

\begin{proposition}[Uniform decay below the continuous spectral bound]
\label{prop:uniformspectraldecay}
For every $0<\nu<-s(A)$ there are $h_0\in(0,\min\{1,L\}]$, $M_\nu\ge1$
and $C_\nu>0$, depending only on $L,\kappa,\nu$, such that for
$h\le h_0$ and $t\ge0$,
\begin{equation}\label{eq:uniformspectraldecay}
 \norm{T_h(t)}_{H\to H},\ \norm{T(t)}_{H\to H}\le M_\nu e^{-\nu t}.
\end{equation}
For $h\le h_0$ and $t>0$, the continuous and discrete semigroups satisfy
\[
 \norm{B_\bullet T_\bullet(t)}\le C_\nu t_\wedge^{-3/2}e^{-\nu t},\qquad
 |T_\bullet(t)x|_1\le C_\nu t_\wedge^{-3/4}e^{-\nu t}\norm x_H ,
\]
where the bullet denotes either operator. Moreover,
\begin{equation}\label{eq:spectralKato}
 \int_0^\infty e^{2\nu t}
 \bigl(|T(t)f|_1^2+|T_h(t)\Pi_hf|_1^2\bigr)\,dt
 \le C_\nu\norm f_H^2 ,
\end{equation}
and the continuous and discrete forced equations have the common estimate
\begin{equation}\label{eq:spectralforcing}
 \E_\nu(y)\le C_\nu\left(\norm{y(0)}_H+
 \norm{e^{\nu t}g}_{L^2_t(V_{1,\bullet}^*)}\right),
\end{equation}
where $V_{1,\bullet}^*$ stands for $V_1^*$ in the continuous problem and for $X_h$ with the norm
$\norm\cdot_{-1,h}$ in the discrete problem. Consequently the continuous
and discrete nonlinear solution maps are analytic on a common small $H$
ball into $\X_{0,\nu}$, with common Taylor remainder bounds.
\end{proposition}
\begin{proof}
\emph{Step 1: resolvent convergence on compact sets.} Choose
$\nu<\nu_1<\gamma<-s(A)$. Put $J=B^{-1}$ and $J_h=B_h^{-1}\Pi_h$. By Lemma~\ref{lem:resdefect} at $z=0$ and
\eqref{eq:approx}, $\norm{J_h-J}\le Ch$.
The operators $J_h$ act as zero on $X_h^\perp$. On $H$,
$R(z,A)=(I+zJ)^{-1}J$ and $R(z,A_h)\Pi_h=(I+zJ_h)^{-1}J_h$. Indeed,
$z-A=(I+zJ)B$. Moreover, $I+zJ_h$ acts as the identity on $X_h^\perp$. For
each compact $Z\subset\rho(A)$,
$\sup_Z\norm{(I+zJ)^{-1}}<\infty$, since $(I+zJ)^{-1}=I-zR(z,A)$. Fix
$h_Z>0$ with $Ch_Z\sup_{z\in Z}|z|\norm{(I+zJ)^{-1}}\le\frac12$. For
every mesh with $h\le h_Z$, the Neumann series and $\norm{J_h-J}\le Ch$
make $I+zJ_h$ invertible on $Z$, with
$\norm{(I+zJ_h)^{-1}}\le2\norm{(I+zJ)^{-1}}$.

\emph{Step 2: a uniform resolvent bound.} By the construction in Lemma~\ref{lem:resdefect}, we have
$\norm{R(-\sigma+i\omega,A_\bullet)}\le2C_R\langle\omega\rangle^{-2/3}$
whenever $0\le\sigma\le c_\Gamma\langle\omega\rangle^{2/3}$. Choose
$\Omega$ with $\gamma\le c_\Gamma\langle\Omega\rangle^{2/3}$. This
controls the strip $-\gamma\le\Re z\le0$ for $|\Im z|\ge\Omega$. The
remaining rectangle is a compact subset of $\rho(A)$, because
$-\gamma>s(A)$. Let $Z$ be this rectangle and
$h_0=\min\{h_Z,1,L\}$. Step~1 controls $Z$ for $h\le h_0$. For $\Re z\ge0$,
the decay $\norm{T_\bullet(t)}\le e^{-\lambda t}$ yields
$\norm{R(z,A_\bullet)}\le(\Re z+\lambda)^{-1}$. Thus
the half-plane $\Re z\ge-\gamma$ lies in all resolvent sets, with
resolvents bounded uniformly in $h\le h_0$.

\emph{Step 3: uniform decay.} Suppose that no common bound at rate
$\nu_1$ exists. Then there are meshes $h_n\le h_0$ and times $t_n$ with
$e^{\nu_1t_n}\norm{T_{h_n}(t_n)}>n$. We do not need $h_n\to0$,
since Step~2 holds for every mesh with $h\le h_0$. Let $\mathfrak S$ be the direct sum
of the $T_{h_n}$ and $\mathfrak G$ its generator. By \eqref{eq:cp} and
Lemma~\ref{lem:directsum}(ii), $\norm{\mathfrak G\mathfrak S(t)}\le
Ct_\wedge^{-3/2}$ for $t>0$. Thus $\mathfrak S$ is of class $\CP_{2/3}$.
It follows from \eqref{eq:sdg} that $\omega_0(\mathfrak G)=s(\mathfrak G)$. By
Step~2 and Lemma~\ref{lem:directsum}(i), the half-plane $\Re z\ge-\gamma$
lies in $\rho(\mathfrak G)$. Hence $s(\mathfrak G)\le-\gamma<-\nu_1$, and
$\sup_n\sup_te^{\nu_1t}\norm{T_{h_n}(t)}<\infty$, a contradiction.
Alternatively, the same conclusion follows from Step 2 and the Gearhart--Huang--Pr\"uss theorem~\cite{Gearhart,Huang,Pruss}. This
proves \eqref{eq:uniformspectraldecay} with $\nu_1$ in place of $\nu$.
The continuous case is Proposition~\ref{prop:spectral}.

\emph{Step 4: smoothing tails.} For $t\ge1$ we write $B_\bullet T_\bullet(t)=B_\bullet T_\bullet(1/2)T_\bullet(t-1/2)$. By the short-time bound \eqref{eq:cp} and Step 3, we obtain the Crandall--Pazy bound
for $t\ge1$,
in fact at the rate $\nu_1$. The energy inequality
$a|T_\bullet(t)x|_1^2\le\norm{B_\bullet T_\bullet(t)x}_H
\norm{T_\bullet(t)x}_H$ yields the $H^1$ estimate. On $[0,1]$ the energy
identity controls the integral in \eqref{eq:spectralKato}. On
$[1,\infty)$ the pointwise bound at the rate $\nu_1>\nu$ controls it.

\emph{Step 5: weak sources.} Let $z$ lie on $\Re z=-\nu$, and let $C_H$
be a common bound for the resolvents on $H$ there. By the adjoint identity \eqref{eq:adjoint} and, on $X_h$, $\Re\ip{B_h^*x}x\ge a|x|_1^2$, we obtain, for the forward and the adjoint resolvents,
\[
 a|R(z,A_\bullet)f|_1^2\le(C_H+\nu C_H^2)\norm f_H^2 .
\]
Applying this to $R(\bar z,A_\bullet^*)$ and dualizing, we obtain
$\norm{R(z,A_\bullet)g}_H\le C_1\norm g_{V_{1,\bullet}^*}$. For
$u=R(z,A_\bullet)g$ the real part identity yields
$a|u|_1^2\le\norm g_{V_{1,\bullet}^*}|u|_1+\nu C_1^2
\norm g_{V_{1,\bullet}^*}^2$. Thus $R(-\nu+i\tau,A_\bullet)$ maps
$V_{1,\bullet}^*$ into $V_1$, uniformly in $\tau$ and $h$. Only adjoint
stability is used here, not adjoint consistency.

Let $y(0)=0$ and let $g$ be a smooth compactly supported $H$-valued
source. Define $y$ by the causal convolution. After weighting and
extension by zero, its Fourier transform is
$\widehat{e^{\nu\cdot}y}(\tau)=R(-\nu+i\tau,A_\bullet)
\widehat{e^{\nu\cdot}g}(\tau)$. The decay at the rate
$\nu_1>\nu$ makes the weighted kernel integrable, which justifies this. By Plancherel's theorem, we obtain
$\norm{e^{\nu t}y}_{L^2_tV_1}\le C\norm{e^{\nu t}g}_{L^2_t(V_{1,\bullet}^*)}$.
For $v=e^{\nu t}y$ and $f=e^{\nu t}g$ the energy identity on $[0,T]$ yields
\[
 \tfrac12\norm{v(T)}_H^2+a\int_0^T|v|_1^2\,dt
 \le\int_0^T\norm f_{V_{1,\bullet}^*}|v|_1\,dt+\nu\int_0^T\norm v_H^2\,dt .
\]
The last integral is bounded by $p^2\norm v_{L^2V_1}^2$, which is already
controlled. Hence the weighted $H$ supremum is also bounded by the source
norm. No positivity of $a-\nu p^2$ is needed. Approximation preserves
continuity, the zero initial value and weighted decay at infinity. Adding $T_\bullet(t)y(0)$ and using \eqref{eq:spectralKato}, we obtain \eqref{eq:spectralforcing}.

\emph{Step 6: analyticity.} By the weighted weak product estimate and Agmon's inequality, we obtain, with $M_w=\sup_te^{\nu t}\norm{w(t)}_H$,
\[
 \norm{e^{\nu t}F(w,\tilde w)}_{L^2_tV_1^*}^2
 \le CM_wM_{\tilde w}^2\int_0^\infty e^{-\nu t}|w|_1\,dt
 \le C_\nu\E_\nu(w)^2\E_\nu(\tilde w)^2 .
\]
The same holds after restriction to $X_h$. The linear and bilinear
trajectory maps therefore have common bounds. The Catalan recursion of Section~\ref{sec:analyticproof} then yields analyticity on a common ball, with
uniqueness and uniform Taylor remainders. The proof is complete.
\end{proof}

\section{Uniform convergence and error estimates in the norm of the
initial datum}\label{sec:samenorm}
\setcounter{equation}{0}

In this section, we define the ordering error and prove the error
estimates in the norm of the datum. The main results are stated in
Section~\ref{sec:mainsame}. The subsections after it contain their
proofs.

\subsection{Two orders of approximation}\label{sec:orders}
Let $u$ solve \eqref{eq:pde}, let $u_h$ solve the weighted scheme
\eqref{eq:Wscheme}, and let
\[
 v=T(\cdot)u_0,\qquad v_h=T_h(\cdot)\Pi_hu_0
\]
be the linear solutions. The linear discrete simulation $v_h$ can be
reached in two orders. Each order divides the total error in its own
way:
\begin{equation}\label{eq:allocation}
 u-v_h=\underbrace{(u-v)}_{\text{linearization}}
      +\underbrace{(v-v_h)}_{\text{discretization}}
      =\underbrace{(u-u_h)}_{\text{discretization}}
      +\underbrace{(u_h-v_h)}_{\text{linearization}} .
\end{equation}
Changing the order changes the linearization share from $u-v$ to
$u_h-v_h$, and the discretization share from $v-v_h$ to $u-u_h$. The two changes cancel in the total error. We call
\begin{equation}\label{eq:defect}
 \Cd_h=(u-v)-(u_h-v_h)=(u-u_h)-(v-v_h)
\end{equation}
the \emph{ordering error}. It measures how much the nonlinear correction
predicted by the discrete model differs from the continuous one.
Equivalently, it measures how much the nonlinearity changes the discretization
error.

Put $z=u-v$, $z_h=u_h-v_h$, $e=u-u_h$ and $m=\frac12(u+u_h)$, and write
\[
 L_h(t)=T_h(t)\Pi_h-T(t),\quad (L_h*g)(t)=\int_0^tL_h(t-s)g(s)\,ds,
 \quad \QQ_h(f,g)(t)=\int_0^tT_h(t-s)F_h(f(s),g(s))\,ds .
\]
Then $v_h-v=L_h(\cdot)u_0$ and, since $\N(u)-\N(u_h)=2F(e,m)$,
\begin{align}
 e&=-L_h(\cdot)u_0+\Cd_h,\qquad e(0)=(I-\Pi_h)u_0,\qquad \Cd_h(0)=0,
 \label{eq:splitting}\\
 \Cd_h&=-L_h*\N(u)+2\QQ_h(e,m).\label{eq:Cdmild}
\end{align}
The linear term $L_h(\cdot)u_0$ is absent from $\Cd_h$. This cancellation
is behind every estimate of the ordering error below. Note that
$L_h(t)u_0=-T(t)(I-\Pi_h)u_0+(T_h(t)-T(t))\Pi_hu_0$ contains the initial
projection error propagated by $T$ and a consistency error.

Throughout, $\varepsilon$ denotes a norm of the datum, and no direction
is fixed. If a direction $\phi$ is fixed and $u_0=\delta\phi$, the ordering error can be read as a
mixed difference. Let $f(\delta,h)=u_h(\delta\phi)$
and $f(\delta,0)=u(\delta\phi)$, put $q(\delta,h)=f(\delta,h)/\delta$
for $\delta\ne0$, and extend $q$ to $\delta=0$ by the derivative of $f$
with respect to $\delta$ at $\delta=0$. Then
\[
 \Cd_h=\delta\{q(\delta,0)-q(0,0)-q(\delta,h)+q(0,h)\}.
\]
Thus $\Cd_h/\delta$ is a mixed second difference of the normalized
solution. Bounds for the four shares give only
$\Cd_h=O(\min\{\delta^2,h^p|\delta|\})$. The product bound
$O(h^p\delta^2)$ requires the cancellation in \eqref{eq:Cdmild}.

\subsection{The main results}\label{sec:mainsame}
In Theorems~\ref{thm:convergence}--\ref{thm:ordering} the error and the
datum are measured in the same norm. The constants depend only on
$L$, $\kappa$ and the decay
rate. We write
$\varepsilon_\theta=\norm{u_0}_{V_\theta}$ for $\theta=0,1,2$. For
sufficiently small $\varepsilon_0$ the four solutions exist for all $t\ge0$, by
Theorem~\ref{thm:continuous} and Proposition~\ref{prop:nonlinstab}. The
results in the norm of the datum fall into three classes: estimates on
the whole half-line $t\ge0$, estimates on $[t_0,\infty)$ for every
$t_0>0$, and estimates for the ordering error. A fourth theorem shows
what a higher norm of the datum adds.

Pointwise rates for the full error cannot hold uniformly on a ball down to
$t=0$. The kernel of $\Pi_h$ contains nonzero smooth compactly supported
functions. Hence $\norm{e(0)}_{V_\theta}=\norm{(I-\Pi_h)u_0}_{V_\theta}$ can
equal the radius of the ball for every $h$. On the whole half-line we
therefore state convergence for each datum, and rates in time-integrated
norms.

\begin{theorem}[The whole half-line]\label{thm:convergence}
Fix $L,\kappa$ and $0<\nu<\lambda$. There are $r,C>0$, depending only
on $L,\kappa,\nu$, such that the following hold for real data and
$0<h\le\min\{1,L\}$.
\begin{enumerate}[label=\rm(\roman*)]
\item If $\varepsilon_0=\norm{u_0}_H\le r$, then
\[
 \norm{e^{\nu t}(v-v_h)}_{L^4(0,\infty;H)}
 +\norm{e^{\nu t}(u-u_h)}_{L^4(0,\infty;H)}\le Ch^{1/6}\varepsilon_0 ,
\]
and, for every fixed such $u_0$,
\[
 \sup_{t\ge0}e^{\nu t}\big(\norm{(v-v_h)(t)}_H+\norm{(u-u_h)(t)}_H\big)
 \longrightarrow0\qquad(h\to0).
\]
\item If $u_0\in V_1$ and $\norm{u_0}_H\le r$, then
\[
 \norm{e^{\nu t}(v-v_h)}_{L^2(0,\infty;H^1)}
 +\norm{e^{\nu t}(u-u_h)}_{L^2(0,\infty;H^1)}\le Ch^{2/3}\varepsilon_1 .
\]
\end{enumerate}
\end{theorem}

\begin{theorem}[Times bounded away from zero]\label{thm:samenorm}
Fix $L,\kappa$ and $0<\nu<\lambda$. There are $r,C>0$, depending only
on $L,\kappa,\nu$, with the following property. Let $\theta\in\{0,1,2\}$
and let $u_0\in V_\theta$ be real with $\varepsilon_\theta\le r$. Then,
for $t>0$,
\begin{align}
 \norm{(v-v_h)(t)}_{V_\theta}
 &\le Ce^{-\nu t}a_\theta(h,t)\,\varepsilon_\theta,\label{eq:sn-linear}\\
 \norm{(u-u_h)(t)}_{V_\theta}
 &\le Ce^{-\nu t}\big\{a_\theta(h,t)\varepsilon_\theta
 +b_\theta(h,t)\varepsilon_\theta^2\big\},\label{eq:sn-full}
\end{align}
where
\[
\begin{array}{c|cc}
 \theta&a_\theta(h,t)&b_\theta(h,t)\\ \hline
 0&\min\{1,ht_\wedge^{-3/2}\}&h^{1/6}\\[1mm]
 1&\min\{h^{1/4}t_\wedge^{-5/8},h^2t_\wedge^{-2}\}&h^{1/6}t_\wedge^{-5/8}
 \\[1mm]
 2&\min\{h^{1/3}t_\wedge^{-1},ht_\wedge^{-3/2}\}&h^{4/15}t_\wedge^{-3/2}
\end{array}.
\]
In particular, for every $t_0>0$, the errors converge to zero uniformly
for $t\ge t_0$ and $\varepsilon_\theta\le r$.
\end{theorem}

\begin{theorem}[The ordering error]\label{thm:ordering}
Fix $L,\kappa$ and $0<\nu<\lambda$. There are $r,C>0$, depending only
on $L,\kappa,\nu$, with the following property. Let $\theta\in\{0,1,2\}$
and let $u_0\in V_\theta$ be real with $\varepsilon_\theta\le r$. The two
linearization shares satisfy
\begin{equation}\label{eq:sn-horizontal}
 \sup_{t\ge0}e^{\nu t}\norm{(u-v)(t)}_{V_\theta}\le C\varepsilon_\theta^2,
 \qquad
 \norm{(u_h-v_h)(t)}_{V_\theta}\le Ce^{-\nu t}c_\theta(t)\,
 \varepsilon_\theta^2,\quad t>0,
\end{equation}
with $c_0=1$ and $c_1=c_2=t_\wedge^{-1/2}$. Their difference, the
ordering error, satisfies $\Cd_h(0)=0$ and
\begin{equation}\label{eq:sn-ordering}
 \norm{\Cd_h(t)}_{V_\theta}\le Ce^{-\nu t}b_\theta(h,t)\,
 \varepsilon_\theta^2,\qquad t>0,
\end{equation}
with $b_\theta$ from Theorem~\ref{thm:samenorm}. At $\theta=0$ the
estimates hold for all $t\ge0$:
$\sup_{t\ge0}e^{\nu t}\norm{\Cd_h(t)}_H\le Ch^{1/6}\varepsilon_0^2$.
Moreover, let $D_h(u_0)=P_2(u_0)-P_{2,h}(\Pi_hu_0)$ be the discretization
error of the quadratic Taylor coefficient. Then
\[
 \sup_{t\ge0}e^{\nu t}\norm{\Cd_h(t)-D_h(u_0)(t)}_H
 \le Ch^{1/6}\varepsilon_0^3,\qquad
 \norm{D_h(u_0)(t)}_H\le Ch^{1/2}t_\wedge^{-5/2}e^{-\nu t}
 \varepsilon_0^2 .
\]
\end{theorem}

The linearization shares are of order $\varepsilon_\theta^2$. Their
difference, the ordering error, carries an additional
mesh factor. In the energy norm this holds uniformly on the whole
half-line, because the linear error $v-v_h$, which contains the initial
projection error, drops out of the mild formula for $\Cd_h$.
Since $|w(x)|\le L^{1/2}\norm{w'}_2$ on $H_0^1$ and
$\norm{w'}_\infty\le C\norm w_{H^2}$, the estimates at $\theta=1$ and
$\theta=2$ give bounds, uniform in $x$, for the errors in the amplitude and in
the slope of the wave. At $\theta=2$ they include the boundary trace
$u_x(0,t)$.

A higher norm of the datum buys uniformity down to $t=0$ in stronger
norms. The next theorem is the only main result in which the right side carries
a higher norm than the left.

\begin{theorem}[Higher norms of the datum on the whole half-line]
\label{thm:higher}
Fix $L,\kappa$ and $0<\nu<\lambda$. There are $r,C>0$, depending only
on $L,\kappa,\nu$, such that the following hold for real data and
$0<h\le\min\{1,L\}$.
\begin{enumerate}[label=\rm(\roman*)]
\item If $u_0\in D(A)$ and $\norm{u_0}_H\le r$, then
\[
 \E_\nu(v-v_h)+\E_\nu(u-u_h)\le Ch^2\norm{u_0}_{H^3}.
\]
\item If $u_0\in V_{3/2}$ and $\varepsilon=\norm{u_0}_{V_{3/2}}\le r$, then
\[
 \sup_{t\ge0}e^{\nu t}\norm{\Cd_h(t)}_{H^1}\le Ch^{1/6}\varepsilon^2,
\]
and, for every fixed such $u_0$,
\[
 \sup_{t\ge0}e^{\nu t}\big(\norm{(v-v_h)(t)}_{H^1}
 +\norm{(u-u_h)(t)}_{H^1}\big)\longrightarrow0\qquad(h\to0).
\]
\end{enumerate}
\end{theorem}
In (i) the constant does not depend on $\norm{u_0}_{H^3}$, which enters
only linearly. Smallness is required only in $H$. In (ii) the datum is
half a derivative smoother than $H_0^1$. The ordering error in $H^1$ then has
the same rate as at the energy level, uniformly on the ball and for all $t\ge0$.

The rates are not claimed to be optimal. Theorems~\ref{thm:convergence}
and~\ref{thm:higher} are proved in Section~\ref{sec:convproof}.
Theorems~\ref{thm:samenorm} and~\ref{thm:ordering} are proved in
Section~\ref{sec:proofsame}.

The bound for the ordering error has a direct computational use. A nonlinear simulation
costs more than a linear one. It is therefore natural to run the two on different meshes.

\begin{corollary}[A nonlinear correction on a coarser mesh]\label{cor:twomesh}
Let $h$ and $\hat h$ be the sizes of two meshes as in
Section~\ref{sec:scheme}, with the same constant $\kappa$. The meshes need not be nested. Let $\theta\in\{0,1,2\}$, let
$u_0\in V_\theta$ be real with $\varepsilon_\theta\le r$, where $r$ is the
smaller of the radii in Theorems~\ref{thm:samenorm} and~\ref{thm:ordering}, and
put $\tilde u=v_h+(u_{\hat h}-v_{\hat h})$. Then, for $t>0$,
\[
 u-\tilde u=(v-v_h)+\Cd_{\hat h},\qquad
 \norm{(u-\tilde u)(t)}_{V_\theta}\le Ce^{-\nu t}
 \big\{a_\theta(h,t)\varepsilon_\theta
 +b_\theta(\hat h,t)\varepsilon_\theta^2\big\}.
\]
\end{corollary}
\begin{proof}
The identity is \eqref{eq:defect} on the mesh of size $\hat h$. By
\eqref{eq:sn-linear} on the mesh of size $h$ and \eqref{eq:sn-ordering} on
the mesh of size $\hat h$, we obtain the estimate.
\end{proof}
This bound has the form of \eqref{eq:sn-full}, with $\hat h$ in place of $h$ in
the second term. The nonlinear mesh enters only at second order in the
datum. For sufficiently small data it may therefore be coarser than the linear one.

\subsection{Linear errors}\label{sec:linearerror}
The theorems are proved from linear error bounds, some of which carry a
higher norm of the datum. We collect them here.

\begin{proposition}[The linear error at positive times]\label{thm:positive}
Let $0\le\theta\le3$, $f\in V_\theta$, $0<\nu<\lambda$ and $t>0$. Then
\[
 \norm{T_h(t)\Pi_hf-T(t)f}_H^2
 =\norm{(I-\Pi_h)T(t)f}_H^2+\norm{T_h(t)\Pi_hf-\Pi_hT(t)f}_H^2,
\]
with
\begin{align*}
 \norm{(I-\Pi_h)T(t)f}_H
 &\le Ch^3t_\wedge^{-(3-\theta)/2}e^{-\lambda t}\norm f_{V_\theta},\\
 \norm{T_h(t)\Pi_hf-\Pi_hT(t)f}_H
 &\le Ch\,\varsigma_\theta(t)\,e^{-\nu t}\norm f_{V_\theta},
 \qquad
 \varsigma_\theta(t)=\begin{cases}t_\wedge^{-(3-\theta)/2},&\theta<3,\\
                  \log(e/t_\wedge),&\theta=3.\end{cases}
\end{align*}
In particular
$\norm{L_h(t)}_{H\to H}\le Cht_\wedge^{-3/2}e^{-\nu t}$.
\end{proposition}
\begin{proof}
The two components are orthogonal. For the first, by \eqref{eq:approx}, we have $\norm{(I-\Pi_h)w}_H\le Ch^3\norm{Aw}_H$ on $D(A)$, and
$\norm{AT(t)B^{-\theta/3}}\le Ct_\wedge^{-(3-\theta)/2}e^{-\lambda t}$ by
Theorem~\ref{thm:ctcp}(ii).

For the second, we use the common contour $\Gamma$ of
Lemma~\ref{lem:CPcontour}. Subtracting the two absolutely convergent representations and using Lemma~\ref{lem:resdefect} with $|z'(\omega)|\le2$, we obtain
\[
 \norm{\{T_h(t)\Pi_h-\Pi_hT(t)\}B^{-\theta/3}}
 \le Ch\int_{\R}e^{-c_\Gamma't\langle\omega\rangle^{2/3}}
 \langle\omega\rangle^{-\theta/3}\,d\omega .
\]
For $0<t\le1$ the substitution $s=t\omega^{2/3}$ shows that the integral is at most
$C\varsigma_\theta(t)$. For $t>1$ let $n=\lceil t\rceil$, $t_1=t/n$, and
telescope
\[
 T_h(t)\Pi_h-\Pi_hT(t)=\sum_{j=0}^{n-1}T_h((n-1-j)t_1)
 \{T_h(t_1)\Pi_h-\Pi_hT(t_1)\}T(jt_1).
\]
Both semigroups decay at the rate $\lambda$, and $T$ commutes with
$B^{\theta/3}$. This yields the factor $ne^{-\lambda(n-1)t_1}\le
Ce^{-\nu t}$.
\end{proof}
The weight $t_\wedge^{-3/2}$ is the time cost of the three derivatives
used in the residual bound \eqref{eq:residual}. For $\theta>3/2$ the
second bound is not sharp in $h$: Lemma~\ref{lem:DAenergy} below gives
$\sup_te^{\nu t}\norm{L_h(t)f}_H\le Ch^{2\theta/3}\norm f_{V_\theta}$.

For smooth compatible data the linear error is of higher order, with a
linear dependence on the datum.

\begin{lemma}[The linear error for smooth data]\label{lem:smoothlinear}
Let $0<\nu<\lambda$ and $f\in D(A^2)$. Then, for $t\ge0$ and $j=0,1,2$,
\[
 \norm{L_h(t)f}_{H^j}\le Ch^{2-j/2}e^{-\nu t}\norm{A^2f}_H ,
\]
and
\[
 \E_\nu(L_h(\cdot)f)\le Ch^2\norm{A^2f}_H,\qquad
 \E_\nu(\partial_tL_h(\cdot)f)\le Ch\norm{A^2f}_H .
\]
\end{lemma}
\begin{proof}
Put $v=T(\cdot)f$. Since $A^2T(t)=T(t)A^2$ and the graph norm of $A^2$ is
equivalent to the $H^6$ norm,
$\norm{v(t)}_{H^6}+\norm{v_t(t)}_{H^3}\le Ce^{-\lambda t}\norm{A^2f}_H$.
Write $-L_h(t)f=v-v_h=\eta_v+\Xi$, with $\eta_v=v-\Pi_hv$ and
$\Xi=\Pi_hv-v_h$. The functions $\Xi$ and $\Xi_t$ satisfy
\[
 \Xi_t=A_h\Xi+d_h(v),\quad\Xi(0)=0,\qquad
 (\Xi_t)_t=A_h\Xi_t+d_h(v_t),\quad\Xi_t(0)=d_h(f).
\]
By \eqref{eq:disforced} and \eqref{eq:residual},
$\E_\nu(\Xi)\le Ch^2\norm{A^2f}_H$ and
$\E_\nu(\Xi_t)\le C(\norm{d_h(f)}_H+h\sup_te^{\nu t}\norm{v_t}_{H^3})
\le Ch\norm{A^2f}_H$. Since $A_h\Xi=\Xi_t-d_h(v)$, it follows from graph recovery \eqref{eq:graph} that $\norm{\Xi(t)}_{H^2}\le Che^{-\nu t}
\norm{A^2f}_H$. Together with \eqref{eq:approx} for $\eta_v$, this proves
the cases $j=0$ and $j=2$. The moment inequality \eqref{eq:moment} on $V_2$ yields the case $j=1$. Finally, by \eqref{eq:approx}, we have $\E_\nu(\eta_v)\le Ch^3\norm{A^2f}_H$ and
$\E_\nu(\partial_t\eta_v)\le Ch^2\norm{A^2f}_H$, since
$\partial_t\eta_v=(I-\Pi_h)v_t$. With the bounds for $\Xi$ and $\Xi_t$
this proves the last two estimates.
\end{proof}

For data in $D(A)$ the residual can be estimated by means of time-integrated
$H^4$ regularity. This requires only three derivatives of the datum.

\begin{lemma}[Energy-trajectory errors for data in $D(A)$]
\label{lem:DAenergy}
Let $0<\nu<\lambda$ and $g\in D(A)$. Then
\begin{equation}\label{eq:H4integrated}
 \sup_{t\ge0}e^{\nu t}\norm{T(t)g}_{H^3}
 +\norm{e^{\nu t}T(\cdot)g}_{L^2_tH^4}\le C\norm{Ag}_H,
 \qquad
 \E_\nu(L_h(\cdot)g)\le Ch^2\norm{Ag}_H .
\end{equation}
Consequently
\begin{equation}\label{eq:Erate}
 \E_\nu(L_h(\cdot)f)\le Ch^{2s/3}\norm f_{V_s},\qquad 0\le s\le3 .
\end{equation}
\end{lemma}
\begin{proof}
Put $v=T(\cdot)g$. Then $Av=T(\cdot)Ag$, and it follows from \eqref{eq:forced} with zero source, applied to $Av$, that
$\int_0^\infty e^{2\nu t}|Av|_1^2\,dt\le C\norm{Ag}_H^2$. Since
$v'''=-Av$ and $v''''=-(Av)'$, this is the integrated $H^4$ bound.
The $H^3$ bound follows from the contraction property. Write $-L_h(\cdot)g=\eta_v+\Xi$ as in
the proof of Lemma~\ref{lem:smoothlinear}. By \eqref{eq:disforced} and
the weak residual bound \eqref{eq:residual} with $s=4$,
$\E_\nu(\Xi)\le Ch^2\norm{e^{\nu t}v}_{L^2_tH^4}\le Ch^2\norm{Ag}_H$.
By \eqref{eq:approx},
$\E_\nu(\eta_v)\le Ch^3(\sup_te^{\nu t}\norm v_{H^3}
+\norm{e^{\nu t}v}_{L^2_tH^4})\le Ch^3\norm{Ag}_H$.

Both linear flows are bounded from $H$ into the energy trajectory space.
Thus the map $f\mapsto L_h(\cdot)f$ is linear and bounded from $H$ and
from $D(A)$ into $\X_0$, with norms at most $C$ and $Ch^2$, respectively. Since $V_s=[H,D(A)]_{s/3}$, \eqref{eq:Erate} follows by interpolation.
\end{proof}

For data in $V_1$ the bound for the linear error in the norm of the
datum is of second order at positive times. For data in $V_2$ it is
of first order. The key is that the resolvent maps $V_s$, $s\ge1$, into
$H^4$.

\begin{proposition}[Linear errors in $V_1$ and $V_2$]\label{prop:H1linear}
Let $0<\nu<\lambda$, $1\le s\le2$ and $t>0$. Then
\begin{align}
 \norm{L_h(t)}_{V_s\to H}&\le Ch^2t_\wedge^{-(4-s)/2}e^{-\nu t},
 \label{eq:physH}\\
 \norm{L_h(t)}_{V_s\to V_1}&\le Ch^2t_\wedge^{-(5-s)/2}e^{-\nu t},
 \label{eq:physH1}\\
 \norm{L_h(t)}_{V_s\to V_2}&\le Ch\,t_\wedge^{-(5-s)/2}e^{-\nu t}.
 \label{eq:physH2}
\end{align}
In particular the same-norm bounds are
$Ch^2t_\wedge^{-2}$ in $V_1$ and $Cht_\wedge^{-3/2}$ in $V_2$.
\end{proposition}
\begin{proof}
Write $L_h(t)f=-(I-\Pi_h)T(t)f+\Theta_h(t)$ with
$\Theta_h(t)=T_h(t)\Pi_hf-\Pi_hT(t)f\in X_h$. By the contour representation of
Lemma~\ref{lem:CPcontour}, $\Theta_h(t)=\frac1{2\pi i}\int_\Gamma
e^{zt}y(z)\,dz$ with $y=R(z,A_h)\Pi_hf-\Pi_hw$, $w=R(z,A)f$, and, as in
the proof of Lemma~\ref{lem:resdefect}, $(z-A_h)y=-r$ with $r=d_h(w)$.
Let $z=-\sigma+i\omega\in\Gamma$ and $\tau=\langle\omega\rangle$. By the adjoint argument in the proof of Lemma~\ref{lem:resdefect}, applied to an arbitrary right side, we have $\norm y_H\le C\tau^{-1/3}\norm r_{-1,h}$. Taking the real part of $\ip{(z-A_h)y}y$, we obtain $a|y|_1^2\le\norm r_{-1,h}|y|_1+\sigma\norm y_H^2$. Since $\sigma\le c_\Gamma\tau^{2/3}$, it follows that
\[
 |y|_1\le C\norm r_{-1,h},\qquad \norm y_H\le C\tau^{-1/3}\norm r_{-1,h}.
\]
For $f\in V_s$, $1\le s\le2$, we have $Bw\in V_1$, hence $w'''\in H^1$ and
$\norm w_{H^4}\le C\norm{B^{4/3}w}_H$. By the moment inequality, with $\norm{R(z,A)}\le C\tau^{-2/3}$ and $\norm{BR(z,A)}\le C\tau^{1/3}$, we obtain
\[
 \norm w_{H^4}\le C\norm{B^{(4-s)/3}R(z,A)}\norm f_{V_s}
 \le C\tau^{(2-s)/3}\norm f_{V_s},
\]
and the weak residual bound in \eqref{eq:residual} for $w\in H^4$ yields
$\norm r_{-1,h}\le Ch^2\tau^{(2-s)/3}\norm f_{V_s}$. For the $H^2$
target we apply the inverse estimate only to $y\in X_h$:
$\norm y_{H^2}\le Ch^{-1}|y|_1$. This avoids the factor $|z|$ that graph
recovery would bring.

For $\beta>-1$ and $0<t\le1$,
$\int_\R e^{-c_\Gamma't\langle\omega\rangle^{2/3}}\langle\omega\rangle^\beta
d\omega\le C_\beta t^{-3(\beta+1)/2}$. The three estimates use
$\beta=(1-s)/3$, $(2-s)/3$ and $(2-s)/3$. The projection terms satisfy
$\norm{(I-\Pi_h)T(t)f}_{H^j}\le Ch^{4-j}\norm{T(t)f}_{H^4}
\le Ch^{4-j}t_\wedge^{-(4-s)/2}e^{-\lambda t}\norm f_{V_s}$, $j=0,1,2$,
and are dominated by the claimed bounds for $t\le1$. For \eqref{eq:physH} and $t>1$ we telescope in $H$ as in the proof
of Proposition~\ref{thm:positive}. For the other estimates and $t\ge2$ we use
the exact identity
\[
 L_h(t)=T_h(1)\Pi_hL_h(t-1)+L_h(1)T(t-1),
\]
together with three ingredients: (a) the bounds for $T_h(1)\Pi_h$ from $H$ into $V_1$ and $V_2$, given by
\eqref{eq:cpH1}, \eqref{eq:cp} and \eqref{eq:graph}; (b) the bounds into $H$
from \eqref{eq:physH} and Proposition~\ref{thm:positive}; (c) the
stability of $T$ on $V_s$. On $1\le t\le2$ we use the contour directly.
\end{proof}

On $X_h$ the norms of $V_\theta$ are controlled by discrete fractional powers
with the exponent $\theta/2$. This gives stability in $H^1$, uniformly
down to $t=0$, for data in $V_{3/2}$.

\begin{lemma}[The norms of $V_\theta$ and discrete fractional powers]
\label{lem:halfbridge}
For $x_h\in X_h$ and $0\le\theta\le2$,
$\norm{x_h}_{V_\theta}\le C\norm{B_h^{\theta/2}x_h}_H$. Consequently
\begin{equation}\label{eq:halfstab}
 \norm{T_h(t)\Pi_hf}_{V_\theta}\le Ce^{-\lambda t}\norm f_{V_{3\theta/2}},
 \qquad t\ge0 .
\end{equation}
\end{lemma}
\begin{proof}
By graph recovery \eqref{eq:graph}, we have $\norm{B_h^{-1}x}_{V_2}\le
C\norm x_H$ on $X_h$. We apply the three-lines lemma to
$z\mapsto e^{c(z-\theta/2)^2}B_h^{-z}$, as a map from
$(X_h,\norm\cdot_H)$ into the couple $(H,V_2)$, with the uniform
imaginary-power bounds for $B_h$. Since $[H,V_2]_{\theta/2}=V_\theta$,
this yields the first assertion. Then, by \eqref{eq:oneway} with
$\alpha=\theta/2$,
$\norm{T_h(t)\Pi_hf}_{V_\theta}\le C\norm{B_h^{\theta/2}T_h(t)\Pi_hf}_H
\le Ce^{-\lambda t}\norm{B^{\theta/2}f}_H$.
\end{proof}

\begin{lemma}[Interpolated linear errors and time kernels]\label{lem:linear}
Fix $0<\nu<\lambda$. For $s=1,2$ and $t>0$,
\begin{align}
 \sup_{t\ge0}e^{\nu t}\norm{L_h(t)f}_H
       &\le Ch^{2s/3}\norm f_{V_s}, \label{eq:linear-H}\\
 \norm{L_h(t)f}_{H^1}
       &\le Ch^{1/4}t_\wedge^{-5/8}e^{-\nu t}\norm f_{V_1},
                 \label{eq:linear-H1}\\
 \norm{L_h(t)f}_{H^2}
       &\le Ch^{1/3}t_\wedge^{-1}e^{-\nu t}\norm f_{V_2}.
                 \label{eq:linear-H2}
\end{align}
Put $\Lambda_h(t)=e^{\nu t}\norm{L_h(t)}_{H\to H}$ and
$\Lambda_h^{(1)}(t)=e^{\nu t}\norm{L_h(t)}_{H\to H^1}$. Then
\begin{equation}\label{eq:kernels}
 \norm{\Lambda_h}_{L^{p_0}(0,\infty)}\le Ch^{2/(3p_0)},\quad 1\le p_0\le4,
 \qquad
 \norm{\Lambda_h^{(1)}}_{L^1(0,\infty)}\le Ch^{1/6}.
\end{equation}
\end{lemma}
\begin{proof}
We interpolate the linear operators $L_h(t)$, not any nonlinear quantity.
The bound \eqref{eq:linear-H} is contained in \eqref{eq:Erate}. By
Lemma~\ref{lem:smoothlinear} and the contraction property, $L_h(t)$ maps $D(B^2)$
to $H^j$ with norm $Ch^{2-j/2}e^{-\nu t}$, and $H$ to $H$ with norm
$2e^{-\lambda t}$. We interpolate with $[H,D(B^2)]_{s/6}=V_s$. For the $H^1$ bound we use the low endpoint
$\norm{L_h(t)}_{H\to H^1}\le Ct_\wedge^{-3/4}e^{-\lambda t}$, from
\eqref{eq:cpH1} and its continuous version, and the high endpoint
$Ch^{3/2}e^{-\nu t}$, with parameter $1/6$. The powers are
$(3/2)(1/6)=1/4$ and $(3/4)(5/6)=5/8$. For $H^2$ we use the low endpoint
$Ct_\wedge^{-3/2}e^{-\lambda t}$, from \eqref{eq:cp} and graph recovery,
the high endpoint $Che^{-\nu t}$, and parameter $1/3$.

Fix $\nu<\nu'<\lambda$. Then
$\Lambda_h(t)\le Ce^{-(\nu'-\nu)t}\min\{1,ht_\wedge^{-3/2}\}$, by the
contraction property and Proposition~\ref{thm:positive}. Splitting the integral at $t=h^{2/3}$ and at $t=1$, we obtain
\[
 \int_0^\infty\Lambda_h\,dt
 \le C\Big(h^{2/3}+h\int_{h^{2/3}}^1t^{-3/2}\,dt+h\Big)
 \le Ch^{2/3}.
\]
Since $\Lambda_h\le2$ by the contraction property, $\int_0^\infty\Lambda_h^{p_0}\,dt\le
2^{p_0-1}\int_0^\infty\Lambda_h\,dt$, which gives the first bound in
\eqref{eq:kernels}. By spatial interpolation between the $H$ error bound and the $H\to H^2$ stability bound, we obtain
$\Lambda_h^{(1)}(t)\le Ce^{-(\nu'-\nu)t}
\min\{t_\wedge^{-3/4},h^{1/2}t_\wedge^{-3/2}\}$. The same splitting gives
$\int_0^{h^{2/3}}t^{-3/4}\,dt=4h^{1/6}$ and
$h^{1/2}\int_{h^{2/3}}^1t^{-3/2}\,dt\le2h^{1/6}$. The tail $t\ge1$
contributes at most $Ch^{1/2}$. This gives the second bound.
\end{proof}

\subsection{The energy ball}\label{sec:energyball}
Throughout this subsection, $0<\nu<\lambda$ is fixed. For $\phi\in H$ put
$\Phi=T(\cdot)\phi$, $\Phi_h=T_h(\cdot)\Pi_h\phi$ and
\begin{equation}\label{eq:Dh}
 D_h(\phi)=\QQ(\Phi,\Phi)-\QQ_h(\Phi_h,\Phi_h)
 =P_2(\phi)-P_{2,h}(\Pi_h\phi),
\end{equation}
the discretization error of the quadratic Taylor coefficient. By
\eqref{eq:lowbilinear} and its discrete analogue, $D_h(\phi)$ is defined
for every $\phi\in H$ and is bounded uniformly in $h$. Its inputs are
linear evolutions. Hence the smoothing of Section~\ref{sec:discrete}
applies to them directly.

\begin{proposition}[The quadratic coefficient]\label{prop:energyquadratic}
For all real or complex $\phi\in H$ and $t>0$,
\begin{equation}\label{eq:energyquadratic}
 \norm{D_h(\phi)(t)}_H\le Ch^{1/2}t_\wedge^{-5/2}e^{-\nu t}\norm\phi_H^2 .
\end{equation}
Moreover, there are $r,C_0>0$, independent of $h$, such that for real
$u_0\in H$ with $\norm{u_0}_H\le r$
\[
 \Cd_h(u_0)(t)=D_h(u_0)(t)+g_h(u_0)(t),\qquad
 \norm{g_h(u_0)(t)}_H\le C_0e^{-\nu t}\norm{u_0}_H^3 .
\]
\end{proposition}
\begin{proof}
Put $R=\norm\phi_H$, $\nu_1=(\lambda+\nu)/2$ and $\Delta_h=\Phi-\Phi_h$.
By Proposition~\ref{thm:positive}, at the indices $0$ and $2$, we have
\[
 \norm{L_h(r')}_{H\to H}\le Chr_\wedge'^{-3/2}e^{-\nu_1r'},\qquad
 \norm{L_h(r')}_{V_2\to H}\le Chr_\wedge'^{-1/2}e^{-\nu_1r'} .
\]
By Theorem~\ref{thm:ctcp}(ii), $\norm{\Phi(s)}_{V_3}\le Cs_\wedge^{-3/2}e^{-\lambda s}R$. The common
energy identity, with $a-\nu_1p^2>0$, yields
\[
 \int_0^\infty e^{2\nu_1s}(|\Phi(s)|_1+|\Phi_h(s)|_1)^2\,ds\le CR^2 .
\]
By symmetry and bilinearity of $F$, we obtain the exact decomposition
\[
 D_h=-\int_0^tL_h(t-s)F(\Phi,\Phi)\,ds
     +\int_0^tT_h(t-s)\Pi_hF(\Delta_h,\Phi+\Phi_h)\,ds=:I_1+I_2 .
\]
We split each integral at $t/2$.

\emph{Early part of $I_1$.} By Agmon's inequality, $\norm{F(\Phi,\Phi)}_H\le C\norm\Phi_H^{1/2}|\Phi|_1^{3/2}$. By H\"older's
inequality in time,
\[
 \int_0^\tau\norm{F(\Phi,\Phi)}_H\,ds\le CR^{1/2}\tau^{1/4}
 \Big(\int_0^\tau|\Phi|_1^2\,ds\Big)^{3/4}\le C\tau^{1/4}R^2 ,
 \qquad 0<\tau\le1 .
\]
The time-integrated gradient enters with the power $3/4$. Hence this bound is uniform on
the ball. For $t\le1$ the early part is at most $Cht^{-5/4}R^2$. For $t\ge1$ the same argument yields
$\int_0^\infty e^{\nu s}\norm{F(\Phi,\Phi)}_H\,ds\le CR^2$. After
extracting $e^{-\nu t}$ from the convolution, the early part is bounded
by $Cht_\wedge^{-5/4}e^{-\nu t}R^2$ for all $t$.

\emph{Late part of $I_1$.} By \eqref{eq:product} at $\theta=3$,
$\norm{F(\Phi,\Phi)}_{V_2}\le Cs_\wedge^{-3}e^{-2\lambda s}R^2$. The
$V_2\to H$ error bound yields
\[
 \norm{I_{1,\mathrm{late}}}_H\le ChR^2\int_{t/2}^t(t-s)_\wedge^{-1/2}
 s_\wedge^{-3}e^{-\nu_1(t-s)-2\lambda s}\,ds
 \le Cht_\wedge^{-5/2}e^{-\nu t}R^2 .
\]

\emph{Early part of $I_2$.} Put $\alpha_h(s)=\min\{1,hs_\wedge^{-3/2}\}$
and $\beta(s)=e^{\nu_1s}(|\Phi(s)|_1+|\Phi_h(s)|_1)$. By the contraction property and the $H\to H$ error bound, $\norm{\Delta_h(s)}_H\le CRe^{-\nu_1s}\alpha_h(s)$. Moreover, by Agmon's inequality, $\norm{\Phi+\Phi_h}_\infty\le CR^{1/2}e^{-\nu_1s}\beta(s)^{1/2}$. Since
$\ip{F(f,g)}w=\frac12\int_0^Lfg\,\overline{(\rho w)'}\dx$,
$\norm{F(\Delta_h,\Phi+\Phi_h)}_{V_1^*}\le C\norm{\Delta_h}_H
\norm{\Phi+\Phi_h}_\infty$. By \eqref{eq:weaksmoothing}, $t-s\ge t/2$ and
H\"older's inequality with exponents $4/3$ and $4$,
\begin{align*}
 \norm{I_{2,\mathrm{early}}}_H
 &\le Ct_\wedge^{-3/4}e^{-\nu t}R^{3/2}
 \Big(\int_0^\infty e^{-\frac43(2\nu_1-\nu)s}\alpha_h(s)^{4/3}ds\Big)^{3/4}
 \Big(\int_0^\infty\beta(s)^2ds\Big)^{1/4}\\
 &\le Ch^{1/2}t_\wedge^{-3/4}e^{-\nu t}R^2 .
\end{align*}
Indeed, for fixed $c_0>0$ and $h\le1$,
\[
 \int_0^\infty e^{-c_0s}\alpha_h(s)^{4/3}\,ds
 \le h^{2/3}+h^{4/3}\int_{h^{2/3}}^1s^{-2}\,ds+C_{c_0}h^{4/3}
 \le C_{c_0}h^{2/3}.
\]
This is the only step that loses half a power of $h$.

\emph{Late part of $I_2$.} By \eqref{eq:cpH1} and Agmon's inequality,
$\norm{\Phi+\Phi_h}_\infty\le Cs_\wedge^{-3/8}e^{-\lambda s/2}R$. With
the $H\to H$ error bound and \eqref{eq:weaksmoothing},
\[
 \norm{I_{2,\mathrm{late}}}_H\le ChR^2\int_{t/2}^t(t-s)_\wedge^{-3/4}
 s_\wedge^{-15/8}e^{-\lambda(t-s)-(\nu_1+\lambda/2)s}\,ds
 \le Cht_\wedge^{-13/8}e^{-\nu t}R^2 .
\]
Adding the four bounds and using $h\le h^{1/2}$, we obtain
\eqref{eq:energyquadratic}.

\emph{Expansion.} By Theorem~\ref{thm:continuous} and
Proposition~\ref{prop:uniformanalytic}, the continuous and discrete
series converge on a common small ball of $H$, with common bounds. The difference of their
quadratic terms is $D_h(u_0)$. The tails from degree three on yield the cubic remainder. The proof is complete.
\end{proof}

The rate $h^{1/2}$ holds for the quadratic coefficient on the whole
energy ball, but the cubic remainder in Proposition~\ref{prop:energyquadratic}
has no mesh factor. Proposition~\ref{thm:energyfull} below supplies one for the complete
ordering error, at a lower rate. Write
\[
 \norm w_{\Z_\nu}=\Big(\int_0^\infty e^{4\nu t}\norm{w(t)}_H^4\,dt\Big)^{1/4}.
\]
All trajectory spaces are built on $H$.
Discrete trajectories are included via $X_h\subset H_0^1\subset H$.

\begin{lemma}[Mixed time norms]\label{lem:energymixed}
For complex trajectories in $\X_0$,
\begin{align}
 \norm w_{\Z_\nu}+\norm w_{L^4_tL^\infty_x}
   &\le C\E_\nu(w),\label{eq:mixedembed}\\
 \norm{e^{\nu t}\N(w)}_{L^{4/3}_tH}
   &\le C\E_\nu(w)^2,\label{eq:mixedsource}\\
 \norm{e^{\nu t}F(f,w)}_{L^2_tV_1^*}
   &\le C\norm f_{\Z_\nu}\E_\nu(w).\label{eq:mixedfeedback}
\end{align}
In the last estimate $f\in\Z_\nu$ suffices. Moreover,
\begin{align}
 \norm{L_h(\cdot)f}_{\Z_\nu}&\le Ch^{1/6}\norm f_H,\label{eq:Jinitial}\\
 \norm{L_h*g}_{\Z_\nu}&\le Ch^{1/3}\norm{e^{\nu t}g}_{L^{4/3}_tH},
 \label{eq:JconvZ}\\
 \norm{L_h*g}_{\Y_\nu(H)}&\le Ch^{1/6}\norm{e^{\nu t}g}_{L^{4/3}_tH}.
 \label{eq:JconvY}
\end{align}
\end{lemma}
\begin{proof}
Let $M_w=\sup_te^{\nu t}\norm{w(t)}_H$ and
$J_w^2=\int_0^\infty e^{2\nu t}|w(t)|_1^2\,dt$. By Poincar\'e's inequality, we have
\[
 \int_0^\infty e^{4\nu t}\norm{w}_H^4\,dt\le
 M_w^2\int_0^\infty e^{2\nu t}\norm w_H^2\,dt\le p^2M_w^2J_w^2 .
\]
By Agmon's inequality, $\norm{w(t)}_\infty^4\le C\norm{w(t)}_H^2|w(t)|_1^2$.
Its unweighted time integral is at most $CM_w^2J_w^2$. For almost every
$t$, $\norm{\N(w(t))}_H\le C\norm{w(t)}_H^{1/2}|w(t)|_1^{3/2}$, hence
\[
 \int_0^\infty e^{4\nu t/3}\norm{\N(w(t))}_H^{4/3}\,dt
 \le CM_w^{2/3}\int_0^\infty e^{2\nu t/3}|w(t)|_1^2\,dt\le CM_w^{2/3}J_w^2 .
\]
Taking the $3/4$ power, we obtain \eqref{eq:mixedsource}. By the weighted weak pairing, $\norm{F(f,w)}_{V_1^*}\le C\norm f_H\norm w_\infty$.
By H\"older's inequality in time, with $\frac12=\frac14+\frac14$, we obtain
\eqref{eq:mixedfeedback}. The whole exponential weight is placed on $f$.
No inverse estimate is used.

For the last three estimates, we note that $e^{\nu t}\norm{L_h(t)f}_H\le
\Lambda_h(t)\norm f_H$. This yields \eqref{eq:Jinitial} by
\eqref{eq:kernels} with $p_0=4$. For the convolutions, we extend the kernel
$\Lambda_h$ and the function $e^{\nu t}\norm{g(t)}_H$ by zero to negative
times. By Young's inequality with $1+\frac14=\frac12+\frac34$, we obtain \eqref{eq:JconvZ}, and with $1=\frac14+\frac34$ we obtain \eqref{eq:JconvY}.
\end{proof}

\begin{proposition}[The complete ordering error on an energy ball]
\label{thm:energyfull}
There are $r,C>0$, independent of $h$, such that every real or complex
$u_0\in H$ with $\varepsilon=\norm{u_0}_H\le r$ satisfies
\begin{align}
 \norm{u-u_h}_{\Z_\nu}&\le Ch^{1/6}\varepsilon,
       \label{eq:energyfullZ}\\
 \norm{\Cd_h(u_0)}_{\Y_\nu(H)}&\le Ch^{1/6}\varepsilon^2,
       \label{eq:energyfulldefect}\\
 \norm{\Cd_h(u_0)-D_h(u_0)}_{\Y_\nu(H)}&\le Ch^{1/6}\varepsilon^3.
       \label{eq:energyfulltail}
\end{align}
The last bound controls the sum of all Taylor terms of degree at least
three. No higher norm of $u_0$ and no $h$-dependent smallness condition
is required.
\end{proposition}
\begin{proof}
By Theorem~\ref{thm:continuous} and
Proposition~\ref{prop:uniformanalytic} there exist a common complex ball
$B_H(0,r_{\rm an})$ and a constant $M_0$, independent of $h$, with
$\E_\nu(u)+\E_\nu(u_h)\le M_0\norm{u_0}_H$. For the discrete map we first
apply the complex-linear contraction $\Pi_h$ to $u_0$ and then embed the
trajectory in the common space $\X_0$. Both maps are holomorphic with
values in this same Banach space. By \eqref{eq:disforced}, we have
$\E_\nu(\QQ_h(y_1,y_2))\le C_F\norm{e^{\nu t}F(y_1,y_2)}_{L^2_tV_1^*}$
whenever the right side is finite. Here $F_h=\Pi_hF$ acts on a weak
source by restriction to $X_h$. Its weak norm does not increase.

It follows from the identities \eqref{eq:splitting} and \eqref{eq:Cdmild} that
\begin{equation}\label{eq:fullid}
 e=-L_h(\cdot)u_0-L_h*\N(u)+2\QQ_h(e,m).
\end{equation}
These are identities between energy solutions. In particular,
$e,m\in\X_0$ and $e\in\Z_\nu$ before any absorption. The source estimate \eqref{eq:mixedsource} yields local $L^1_tH$ integrability. Hence the
mild integrals and the weak forced solutions coincide. By
Lemma~\ref{lem:energymixed},
\[
 \norm e_{\Z_\nu}\le C_1h^{1/6}\varepsilon+C_2h^{1/3}\varepsilon^2
 +C_3\varepsilon\norm e_{\Z_\nu}.
\]
Choose $0<R_0<r_{\rm an}$ with $C_3R_0\le1/2$. For complex
$u_0$ with $\norm{u_0}_H\le R_0$, absorption and $h^{1/3}\le h^{1/6}$ yield
\eqref{eq:energyfullZ}. Then, by \eqref{eq:JconvY} and the forced estimate, applied to \eqref{eq:Cdmild}, we obtain
\[
 \norm{\Cd_h(u_0)}_{\Y_\nu(H)}
 \le Ch^{1/6}\E_\nu(u)^2+C\norm e_{\Z_\nu}\E_\nu(m)
 \le C_4h^{1/6}\varepsilon^2 .
\]
All trajectories are $H$-continuous. Hence these estimates hold at every
time. At zero the ordering error vanishes.

It remains to estimate the Taylor tail. For a complex unit vector
$\phi\in H$, the function $z\mapsto\Cd_h(z\phi)$ is holomorphic with
values in $\Y_\nu(H)$ on a neighborhood of $|z|\le R_0$. Its constant
and linear coefficients vanish, and its quadratic coefficient is
$D_h(\phi)$. Write $\Cd_h(z\phi)=\sum_{n\ge2}z^na_{n,h}(\phi)$. The
preceding estimate holds on the circle $|z|=R_0$. Hence, by the Banach-valued Cauchy formula,
$\norm{a_{n,h}(\phi)}_{\Y_\nu(H)}\le C_4h^{1/6}R_0^{2-n}$ for $n\ge2$. Put
$r=R_0/2$. For $u_0=\varepsilon\phi$ with $0<\varepsilon\le r$,
\[
 \norm{\Cd_h(u_0)-D_h(u_0)}_{\Y_\nu(H)}
 \le C_4h^{1/6}R_0^2\sum_{n\ge3}(\varepsilon/R_0)^n
 \le\frac{2C_4}{R_0}h^{1/6}\varepsilon^3 .
\]
The constants and radii are independent of the direction, of the mesh and
of the number of Taylor terms. The proof is complete.
\end{proof}

The factor $h^{1/6}$ is the $L^4$ size of an initial layer of width $h^{2/3}$.
The proof uses first-order consistency only for strong sources, in
$L^{4/3}_tH$, and only stability for weak sources.

\subsection{Proof of Theorems~\ref{thm:convergence} and~\ref{thm:higher}}
\label{sec:convproof}
In this subsection, we prove Theorems~\ref{thm:convergence}
and~\ref{thm:higher}. The proof is divided into four steps: the
convergence of the linear maps, the convergence in the energy norm, the
rates for data in $D(A)$ and in $V_1$, and the estimates in $H^1$ for
data in $V_{3/2}$.

Write $\mathcal Kg(t)=\int_0^tT(t-s)g(s)\,ds$ and
$\mathcal K_hg(t)=\int_0^tT_h(t-s)\Pi_hg(s)\,ds$ for the forced solutions
with zero initial value. For weak sources they are the energy solutions
of Sections~\ref{sec:analyticproof} and~\ref{sec:stability}. Let
$L^2_\nu(V_1^*)$ be the space of sources with
$\norm g_{L^2_\nu(V_1^*)}=\norm{e^{\nu t}g}_{L^2_tV_1^*}<\infty$. By
\eqref{eq:forced} and \eqref{eq:disforced}, and by the weak product bound
of Section~\ref{sec:analyticproof},
\begin{equation}\label{eq:Kbounds}
 \E_\nu(\mathcal Kg)+\E_\nu(\mathcal K_hg)\le C\norm g_{L^2_\nu(V_1^*)},
 \qquad
 \norm{F(w,\tilde w)}_{L^2_\nu(V_1^*)}\le C_B\E_\nu(w)\E_\nu(\tilde w).
\end{equation}

\emph{Step 1: linear maps.} Both linear flows are bounded from $H$ into
$\X_0$, and $\E_\nu(L_h(\cdot)g)\le Ch^2\norm{Ag}_H$ for $g\in D(A)$, by
Lemma~\ref{lem:DAenergy}. Since $D(A)$ is dense in $H$,
$\E_\nu(v-v_h)=\E_\nu(L_h(\cdot)u_0)\to0$ for every $u_0\in H$. The same
argument applies to forced solutions. For $G\in C_c^\infty((0,\infty);D(A))$,
$(\mathcal K-\mathcal K_h)G(t)=-\int_0^tL_h(t-s)G(s)\,ds$. A time shift
by $s$ multiplies the weighted norm by $e^{\nu s}$. Hence, by Minkowski's inequality and Lemma~\ref{lem:DAenergy}, we obtain
\[
 \E_\nu((\mathcal K-\mathcal K_h)G)
 \le Ch^2\int_0^\infty e^{\nu s}\norm{AG(s)}_H\,ds .
\]
Such sources are dense in $L^2_\nu(V_1^*)$. Indeed, we truncate and mollify in time
and then approximate in space, since $D(A)$ is dense in $H$ and $H$ is
dense in $V_1^*$. By the uniform bound \eqref{eq:Kbounds},
\begin{equation}\label{eq:forcedstrong}
 \E_\nu((\mathcal K-\mathcal K_h)G)\longrightarrow0\qquad
 \text{for each fixed }G\in L^2_\nu(V_1^*).
\end{equation}

\emph{Step 2: convergence in the energy norm.} By
Theorem~\ref{thm:continuous} and Proposition~\ref{prop:uniformanalytic},
$\E_\nu(u)+\E_\nu(u_h)\le M_0\norm{u_0}_H$ for $\norm{u_0}_H\le r$.
Subtracting the two mild equations and using $\N(u)-\N(u_h)=2F(e,m)$, we obtain
\[
 e=-L_h(\cdot)u_0+(\mathcal K-\mathcal K_h)\N(u)+2\mathcal K_hF(e,m).
\]
By \eqref{eq:Kbounds}, the last term has energy norm at most
$CC_BM_0r\,\E_\nu(e)\le\frac12\E_\nu(e)$ for a small fixed $r$. Since
$\E_\nu(e)$ is finite,
\[
 \E_\nu(e)\le2\E_\nu(L_h(\cdot)u_0)
 +2\E_\nu((\mathcal K-\mathcal K_h)\N(u))\longrightarrow0,
\]
by Step 1 and \eqref{eq:forcedstrong} with the fixed source
$\N(u)\in L^2_\nu(V_1^*)$. This proves the convergence in Theorem~\ref{thm:convergence}(i). No error
estimate for the nonlinear problem has been used. The $L^4$ bound in Theorem~\ref{thm:convergence}(i)
follows from \eqref{eq:Jinitial} and \eqref{eq:energyfullZ}.

\emph{Step 3: Theorem~\ref{thm:higher}(i) and
Theorem~\ref{thm:convergence}(ii).} First let $g\in D(A)$ be real with
$\norm g_H\le r$, and let $u_0=g$. Its $H^3$ norm may be large. The
constants below do not depend on it. We first bound the discrete
solutions, which solve ordinary differential equations, and then let
$h\to0$.

The right side of \eqref{eq:Wscheme} is a quadratic polynomial on $X_h$.
Hence $u_h$ is smooth in time, and $q_h=\dot u_h$ solves the linear
equation
\[
 \dot q_h=A_hq_h+2F_h(u_h,q_h),\qquad
 q_h(0)=A_h\Pi_hg+\Pi_h\N(\Pi_hg).
\]
By \eqref{eq:graphproj}, $\norm{A_h\Pi_hg}_H\le C\norm g_{H^3}$. By
\eqref{eq:prodenergy} and \eqref{eq:disenergy},
$\norm{\N(\Pi_hg)}_H\le C|\Pi_hg|_1^2\le C\norm g_H\norm{A_h\Pi_hg}_H$.
Hence $\norm{q_h(0)}_H\le C\norm g_{H^3}$. On $\X_0$ consider the affine
map $\Phi(q)=T_h(\cdot)q_h(0)+2\mathcal K_hF(u_h,q)$. By \eqref{eq:Kbounds}
and Step~2, its linear part has norm at most $2CC_BM_0\norm g_H$, which
is at most $1/2$ for a small fixed $r$. Thus $\Phi$ is a contraction.
Its fixed point solves the same linear equation with the same initial
value as $q_h$. By uniqueness for linear ordinary differential equations,
the fixed point is $q_h$. Hence $\E_\nu(\dot u_h)\le C\norm g_{H^3}$.
The scheme also gives $B_hu_h=\Pi_h\N(u_h)-q_h$. As above,
$\norm{\N(u_h)}_H\le C\norm{u_h}_H\norm{B_hu_h}_H$, and
$\norm{u_h(t)}_H\le\norm g_H$ by Proposition~\ref{prop:nonlinstab}. For
small $r$ this term is absorbed, and
$\norm{A_hu_h(t)}_H\le2\norm{q_h(t)}_H\le Ce^{-\nu t}\norm g_{H^3}$.

Now let $h\to0$. By Step~2, $\E_\nu(u-u_h)\to0$. By \eqref{eq:graph},
$\norm{u_h(t)}_{H^2}\le Ce^{-\nu t}\norm g_{H^3}$, so $u_h(t)\to u(t)$
weakly in $H^2$. For $\varphi\in C_c^\infty(0,L)$ we have
$\Pi_h\varphi\to\varphi$ in $H^1$ by \eqref{eq:approx}, and
\[
 |b(u_h(t),\Pi_h\varphi)|=|\ip{A_hu_h(t)}{\Pi_h\varphi}|
 \le Ce^{-\nu t}\norm g_{H^3}\norm\varphi_H .
\]
In the limit, $|\ip{u(t)'''}\varphi|\le Ce^{-\nu t}\norm g_{H^3}\norm\varphi_H$.
Hence $u(t)\in D(A)$ and $\norm{u(t)}_{H^3}\le Ce^{-\nu t}\norm g_{H^3}$.
After extraction of a subsequence, $e^{\nu t}\dot u_h$ converges weak-star
in $L^\infty(0,\infty;H)$ and weakly in $L^2(0,\infty;H_0^1)$ to
$e^{\nu t}q$. Passing to the limit in $u_h(t)-u_h(s)=\int_s^t\dot u_h\,dr$
gives $q=\dot u$. By weak lower semicontinuity,
$\E_\nu(\dot u)\le C\norm g_{H^3}$, where the supremum is at first an
essential supremum. Since $Au$ is bounded in $H$, the mild equation gives
$\dot u=Au+\N(u)$ almost everywhere. Moreover
$\dot u\in L^2_{\rm loc}(H_0^1)$, so $\N(u)$ is absolutely continuous in
$H$ with derivative $2F(u,\dot u)$, by \eqref{eq:prodenergy}.
Differentiating the mild equation, we see that $\dot u$ is the energy
solution of $\dot q=Aq+2F(u,q)$ with $q(0)=Ag+\N(g)$. In particular
$\dot u$ is continuous in $H$, and the essential supremum is a supremum.
Thus $u$ is a global strong solution, and $u'''=Bu=\N(u)-\dot u$. Since
$\norm{\N(u)'}_2\le C|u|_1\norm u_{H^3}$ and
$\norm w_{H^4}\le C(\norm w_H+\norm{w''''}_2)$, differentiating this
identity once in $x$ yields the integrated $H^4$ bound, as in the proof
of Lemma~\ref{lem:DAenergy}. Hence
\begin{equation}\label{eq:nonlinH4}
 \sup_{t\ge0}e^{\nu t}\norm{u(t)}_{H^3}
 +\norm{e^{\nu t}u}_{L^2_tH^4}\le C\norm g_{H^3}.
\end{equation}
Now write $\zeta=u-\Pi_hu$ and $y_h=\Pi_hu-u_h$. Then
$\dot y_h=A_hy_h+2F_h(m,y_h)+d_h(u)+2F_h(m,\zeta)$, $y_h(0)=0$. By
\eqref{eq:approx} and \eqref{eq:nonlinH4},
$\E_\nu(\zeta)\le Ch^3\norm g_{H^3}$. The residual bound \eqref{eq:residual} with $s=4$ gives
$\norm{e^{\nu t}d_h(u)}_{L^2_t(-1,h)}\le Ch^2\norm g_{H^3}$, and
\eqref{eq:prodenergy} gives
$\norm{e^{\nu t}F_h(m,\zeta)}_{L^2_t(-1,h)}\le
C\sup_te^{\nu t}\norm\zeta_H\norm{|m|_1}_{L^2_t}\le Ch^3\norm g_{H^3}$.
It then follows from \eqref{eq:disforced}, \eqref{eq:Kbounds} and Step~2 that
$\E_\nu(y_h)\le Ch^2\norm g_{H^3}+2CC_BM_0\norm g_H\E_\nu(y_h)$. Here
$\E_\nu(y_h)$ is finite, since $\Pi_h$ is bounded on $V_1$. For small
$r$ the last term is absorbed. With the bound for $\zeta$, we obtain
\begin{equation}\label{eq:DAnonlinear}
 \E_\nu(u-u_h)\le Ch^2\norm g_{H^3}.
\end{equation}
Together with Lemma~\ref{lem:DAenergy} for $v-v_h$, this proves
Theorem~\ref{thm:higher}(i).

For $u_0\in V_1$ with $\norm{u_0}_H\le r$ and $Q\ge1$, put
$g_Q=Q(Q+B)^{-1}u_0$. Since $B$ is maximal accretive, $Q(Q+B)^{-1}$ is a
contraction on $H$. Since $B$ commutes with $(Q+B)^{-1}$, the moment inequality gives
\[
 \norm{g_Q}_H\le\norm{u_0}_H,\qquad
 \norm{u_0-g_Q}_H\le CQ^{-1/3}\norm{u_0}_{V_1},\qquad
 \norm{g_Q}_{H^3}\le CQ^{2/3}\norm{u_0}_{V_1},
\]
since $u_0-g_Q=B^{2/3}(Q+B)^{-1}B^{1/3}u_0$. On the small energy ball the
continuous and discrete solution maps are Lipschitz in $\E_\nu$. Indeed,
the difference $w$ of two solutions satisfies
$w=T(\cdot)w(0)+2\mathcal KF(w,\bar m)$, where $\bar m$ is their mean.
As in Step~2, the last term is absorbed. The discrete maps are treated
in the same way. Moreover, $\Pi_h$ is a contraction on $H$. Hence, by \eqref{eq:DAnonlinear} for
$g_Q$,
\[
 \E_\nu(u-u_h)\le C\big(Q^{-1/3}+h^2Q^{2/3}\big)\norm{u_0}_{V_1}.
\]
The choice $Q=h^{-2}$ yields $\E_\nu(u-u_h)\le Ch^{2/3}\norm{u_0}_{V_1}$.
By \eqref{eq:Erate} with $s=1$, we obtain the same bound for $v-v_h$. Since
$\E_\nu$ dominates the weighted $L^2_tH^1$ norm, this proves
Theorem~\ref{thm:convergence}(ii).

\emph{Step 4: Theorem~\ref{thm:higher}(ii).} Let
$\varepsilon=\norm{u_0}_{V_{3/2}}\le r$ and put
$\nu_2=\nu+\frac23(\lambda-\nu)$, so that $\nu<\nu_2<\lambda$. By
Theorem~\ref{thm:continuous} with $\theta=3/2$,
$\norm{u(t)}_{H^1}\le C\varepsilon e^{-\nu_2t}$. For the discrete solution, by \eqref{eq:halfstab}, we have $\norm{T_h(t)\Pi_hu_0}_{H^1}\le Ce^{-\lambda t}\varepsilon$. Moreover, by \eqref{eq:cpH1}, $\norm{T_h(t)\Pi_hg}_{H^1}\le
Ct_\wedge^{-3/4}e^{-\lambda t}\norm g_H$. Since
$\norm{\N(w)}_H\le C\norm w_{H^1}^2$, the map defined by the mild equation is a contraction
on a small ball of $\Y_{\nu_2}(H^1)$. Its fixed point is $u_h$, and
$\norm{u_h(t)}_{H^1}\le C\varepsilon e^{-\nu_2t}$.

As in Step 2,
$e=-L_h(\cdot)u_0+(\mathcal K-\mathcal K_h)\N(u)+2\mathcal K_hF(e,m)$. In
$\Y_\nu(H^1)$ the last term is at most
$C\varepsilon\norm e_{\Y_\nu(H^1)}$, by \eqref{eq:cpH1} and
$\norm{F(e,m)}_H\le C\norm e_{H^1}\norm m_{H^1}$. It is absorbed. The
second term is at most
$\norm{\Lambda_h^{(1)}}_{L^1}\sup_se^{\nu s}\norm{\N(u(s))}_H\le
Ch^{1/6}\varepsilon^2$, by \eqref{eq:kernels}. For the first term,
$\norm{L_h(t)g}_{H^1}\le Ch^{3/2}e^{-\nu t}\norm{A^2g}_H$ for
$g\in D(A^2)$ by Lemma~\ref{lem:smoothlinear}. Moreover,
$\norm{L_h(t)(u_0-g)}_{H^1}\le Ce^{-\lambda t}\norm{u_0-g}_{V_{3/2}}$ by
\eqref{eq:halfstab} and the stability of the continuous semigroup. The functions $g_n=[n(n+B)^{-1}]^2u_0$ lie in
$D(A^2)$ and tend to $u_0$ in $V_{3/2}=D(B^{1/2})$: the operators
$n(n+B)^{-1}$ are uniformly bounded, commute with $B^{1/2}$ and tend
strongly to the identity. Taking first $n$ large and then $h$ small, we obtain
$\norm{L_h(\cdot)u_0}_{\Y_\nu(H^1)}\to0$. Hence
$\norm e_{\Y_\nu(H^1)}\to0$. The first term is $v-v_h$, so the linear
error converges as well.

For the ordering error, it follows from \eqref{eq:Cdmild} and $e=-L_h(\cdot)u_0+\Cd_h$ that
\[
 \norm{\Cd_h}_{\Y_\nu(H^1)}\le Ch^{1/6}\varepsilon^2
 +C\varepsilon\norm{\Cd_h}_{\Y_\nu(H^1)}
 +2\norm{\mathcal K_hF(L_h(\cdot)u_0,m)}_{\Y_\nu(H^1)} .
\]
By \eqref{eq:halfstab}, the stability of the continuous semigroup and \eqref{eq:physH1}
with regularity index $3/2$,
\[
 \norm{L_h(s)u_0}_{H^1}
 \le C\phi_h(s)e^{-\nu s}\varepsilon,\qquad
 \phi_h(s)=\min\{1,h^2s_\wedge^{-7/4}\}.
\]
Put $c=\lambda-\nu>0$ and $c'=\nu_2>0$. By the bound on $m$, the $H^1$ product estimate and \eqref{eq:cpH1}, we obtain
\[
 e^{\nu t}\norm{\mathcal K_hF(L_h(\cdot)u_0,m)(t)}_{H^1}
 \le C\varepsilon^2 I_h(t),\qquad
 I_h(t)=\int_0^t(t-s)_\wedge^{-3/4}e^{-c(t-s)}\phi_h(s)e^{-c's}\,ds .
\]
Set $\ell=h^{8/7}\le1$, so that $h^2=\ell^{7/4}$. If $s_\wedge\le\ell$,
then $\phi_h(s)\le1\le(\ell/s_\wedge)^{1/4}$. If $s_\wedge\ge\ell$, then
$h^2s_\wedge^{-7/4}=\ell^{7/4}s_\wedge^{-3/2}s_\wedge^{-1/4}
\le\ell^{1/4}s_\wedge^{-1/4}$. Hence
\[
 \phi_h(s)\le h^{2/7}s_\wedge^{-1/4},\qquad s>0 .
\]
Since $r_\wedge^{-3/4}e^{-cr}\le r^{-3/4}\mathbf 1_{r<1}+e^{-cr}$ and
$s_\wedge^{-1/4}\le s^{-1/4}\mathbf 1_{s<1}+1$, we expand the product
and use $\int_0^t(t-s)^{-3/4}s^{-1/4}\,ds=\mathrm B(\frac14,\frac34)$. This
yields
\[
 \sup_{t>0}I_h(t)\le\Big(\mathrm B\big(\tfrac14,\tfrac34\big)+\tfrac{16}3+\tfrac1c\Big)
 h^{2/7}.
\]
Before absorption, the norm of $\Cd_h$ in $\Y_\nu(H^1)$ is finite: the
bounds already proved put $u,u_h,v,v_h$ in this space, and
$\Cd_h=(u-v)-(u_h-v_h)$. Thus the preceding inequality becomes
\[
 \norm{\Cd_h}_{\Y_\nu(H^1)}
 \le C(h^{1/6}+h^{2/7})\varepsilon^2
 +C\varepsilon\norm{\Cd_h}_{\Y_\nu(H^1)} .
\]
A fixed reduction of the radius makes $C\varepsilon\le1/2$. Since
$h^{2/7}\le h^{1/6}$ for $0<h\le1$, this proves
$\norm{\Cd_h}_{\Y_\nu(H^1)}\le Ch^{1/6}\varepsilon^2$. The trajectories
are continuous in $H^1$ and $\Cd_h(0)=0$. Hence the bound includes
$t=0$. This completes the proof of Theorems~\ref{thm:convergence}
and~\ref{thm:higher}.\hfill$\square$

\subsection{\texorpdfstring{The norms of $H_0^1$ and $V_2$}{The norms of H\_0\textasciicircum 1 and V\_2}}\label{sec:strong}
In this subsection, we prove the nonlinear estimates of
Theorems~\ref{thm:samenorm} and~\ref{thm:ordering} at $\theta=1$ and
$\theta=2$. The proof is divided into several lemmas. The estimates at
$\theta=1$ are proved after Lemma~\ref{lem:H1stable}. Those at $\theta=2$
are proved after Lemma~\ref{lem:holder}.

Put $d=\lambda-\nu>0$, $\nu_1=\nu+d/3$ and $\nu_2=\nu+2d/3$, so that
$\nu<\nu_1<\nu_2<\lambda$. For $s=1,2$ put
$\varepsilon_s=\norm{u_0}_{V_s}$, and assume $\varepsilon_s\le r$ with a
radius $r$ that may be reduced finitely many times, each time to a value depending
only on $L,\kappa$ and the fixed rates. By Theorem~\ref{thm:continuous}
and Proposition~\ref{prop:uniformanalytic},
\begin{equation}\label{eq:cts-strong}
 \norm{u(t)}_{V_s}\le C\varepsilon_se^{-\nu_2t},\qquad
 \E_{\nu_2}(u)+\E_{\nu_2}(u_h)\le C\varepsilon_s .
\end{equation}
Time derivatives are written with dots. Introduce the discrete linear
flow driven by the continuous nonlinearity,
\[
 \tilde u_h=v_h+\int_0^tT_h(t-s)\Pi_h\N(u(s))\,ds,\qquad
 \tilde e=\tilde u_h-u=L_h(\cdot)u_0+S_h,\qquad S_h=L_h*\N(u),
\]
and put $y=u_h-\tilde u_h$. Then $e=-\tilde e-y$, $\Cd_h=-S_h-y$, and
\begin{equation}\label{eq:yequation}
 y_t=A_hy+2F_h(m,y)+2F_h(m,\tilde e),\qquad y(0)=0 .
\end{equation}

\begin{lemma}[A comparison in the energy norm]\label{lem:prelim}
For $s=1,2$,
\[
 \sup_{t\ge0}e^{\nu_1t}\norm{e(t)}_H\le Ch^{2/3}\varepsilon_s,\qquad
 \sup_{t\ge0}e^{\nu_1t}\norm{\Cd_h(t)}_H\le Ch^{2/3}\varepsilon_s^2 .
\]
\end{lemma}
\begin{proof}
By \eqref{eq:cts-strong},
$\norm{\N(u(s))}_H\le C\varepsilon_s^2e^{-2\nu_2s}$. By Lemma~\ref{lem:linear} at the rate $\nu_1$, we have
$\sup_te^{\nu_1t}\norm{S_h(t)}_H\le\norm{\Lambda_h}_{L^1}
\sup_se^{\nu_1s}\norm{\N(u(s))}_H\le Ch^{2/3}\varepsilon_s^2$, and
$\sup_te^{\nu_1t}\norm{L_h(t)u_0}_H\le Ch^{2s/3}\varepsilon_s$. Hence
$\sup_te^{\nu_1t}\norm{\tilde e}_H\le C(h^{2s/3}\varepsilon_s+
h^{2/3}\varepsilon_s^2)\le Ch^{2/3}\varepsilon_s$. By \eqref{eq:disforced}
at the rate $\nu_1$ and the weak product bound in \eqref{eq:prodenergy},
\[
 \E_{\nu_1}(y)\le C\Big(\int_0^\infty e^{2\nu_1t}|m|_1^2
 \norm{\tilde e}_H^2\,dt\Big)^{1/2}
 +C\norm{|m|_1}_{L^2_t}\E_{\nu_1}(y).
\]
By \eqref{eq:cts-strong}, $\norm{|m|_1}_{L^2_t}\le C\varepsilon_s$. Here
$\E_{\nu_1}(y)$ is finite, since $\E_{\nu_1}(u_h)$ and
$\E_{\nu_1}(\tilde u_h)$ are, by \eqref{eq:cts-strong},
\eqref{eq:disforced} and \eqref{eq:Kbounds}. Hence the last term is
absorbed for small $\varepsilon_s$, and
$\E_{\nu_1}(y)\le C\varepsilon_s\sup_te^{\nu_1t}\norm{\tilde e(t)}_H$.
Since $e=-\tilde e-y$ and $\Cd_h=-S_h-y$, the claim follows.
\end{proof}

\begin{lemma}[Discrete $H^1$ stability at positive times]\label{lem:H1stable}
If $\varepsilon_1$ is sufficiently small, then
\[
 \norm{u_h(t)}_{H^1}\le C\varepsilon_1t_\wedge^{-1/2}e^{-\nu_1t},\qquad
 t>0 .
\]
\end{lemma}
\begin{proof}
By \eqref{eq:oneway}, the moment inequality and \eqref{eq:cp},
\[
 \norm{B_hT_h(t)\Pi_hu_0}_H\le\norm{B_h^{2/3}T_h(t)}
 \norm{B_h^{1/3}\Pi_hu_0}_H\le Ct_\wedge^{-1}e^{-\lambda t}\varepsilon_1 .
\]
The energy inequality in \eqref{eq:disenergy} therefore bounds the linear
solution in $H^1$ by $Ct_\wedge^{-1/2}e^{-\lambda t}\varepsilon_1$. By Agmon's inequality, $\norm{\N(w)}_H\le C\norm w_H^{1/2}\norm w_{H^1}^{3/2}$.
For a finite horizon $T_0$ put
$M_1=\sup_{0<t\le T_0}e^{\nu_1t}t_\wedge^{1/2}\norm{u_h(t)}_{H^1}$. By the mild equation, \eqref{eq:cpH1} and \eqref{eq:cts-strong}, we obtain
\[
 M_1\le C\varepsilon_1+C\varepsilon_1^{1/2}M_1^{3/2}.
\]
Near zero the convolution used is
$\int_0^t(t-s)^{-3/4}s^{-3/4}\,ds=\mathrm B(\frac14,\frac14)t^{-1/2}$.
For each fixed mesh the weighted quantity tends to zero as
$t\downarrow0$, since $X_h$ is finite dimensional. A continuity
argument at the level $M_1=2C\varepsilon_1$, with $\varepsilon_1$ small, yields the bound on every horizon.
\end{proof}

\emph{Proof of the estimates at $\theta=1$.} For $e=u-u_h$, by Agmon's inequality, \eqref{eq:cts-strong} and Lemma~\ref{lem:H1stable}, we obtain
\begin{equation}\label{eq:strongproduct}
 \norm{F(e,m)}_H\le C\varepsilon_1e^{-\nu_1s}
 \big\{s_\wedge^{-1/4}\norm e_{H^1}+s_\wedge^{-3/4}\norm e_H\big\}.
\end{equation}
Indeed, $\norm{e'm}_2\le\norm m_\infty\norm{e'}_2$ with
$\norm m_\infty\le C\varepsilon_1s_\wedge^{-1/4}e^{-\nu_1s}$, and
\[
 \norm{em'}_2\le\norm e_\infty\norm{m'}_2\le C\varepsilon_1
 s_\wedge^{-1/2}e^{-\nu_1s}\norm e_H^{1/2}\norm e_{H^1}^{1/2},
\]
to which Young's inequality with these time weights applies.

By \eqref{eq:Cdmild}, $\Cd_h=-S_h+2\QQ_h(e,m)$. By \eqref{eq:kernels},
\[
 \sup_te^{\nu t}\norm{S_h(t)}_{H^1}\le\norm{\Lambda_h^{(1)}}_{L^1}
 \sup_se^{\nu s}\norm{\N(u(s))}_H\le Ch^{1/6}\varepsilon_1^2 .
\]
In
$\QQ_h(e,m)$ we use \eqref{eq:cpH1} and \eqref{eq:strongproduct}, and insert
$e=-L_h(\cdot)u_0+\Cd_h$ into the first term of \eqref{eq:strongproduct}. By \eqref{eq:linear-H1} and
Lemma~\ref{lem:prelim}, the two known contributions involve the
convolutions
\[
 \int_0^t(t-s)^{-3/4}s^{-7/8}\,ds=Ct^{-5/8},\qquad
 \int_0^t(t-s)^{-3/4}s^{-3/4}\,ds=Ct^{-1/2},
\]
with the factors $h^{1/4}\varepsilon_1^2$ and $h^{2/3}\varepsilon_1^2$.
Both endpoint singularities are integrable. The quantity
\[
 X_1=\sup_{0<t\le T_0}e^{\nu t}t_\wedge^{5/8}\norm{\Cd_h(t)}_{H^1}
\]
is finite. Indeed, $\Cd_h=e+L_h(\cdot)u_0$, and \eqref{eq:cts-strong},
Lemma~\ref{lem:H1stable} and \eqref{eq:linear-H1} bound the $H^1$ norms
of $u(t)$, $u_h(t)$ and $L_h(t)u_0$ by
$C\varepsilon_1t_\wedge^{-5/8}e^{-\nu t}$. Moreover
\[
 X_1\le Ch^{1/6}\varepsilon_1^2+C\varepsilon_1X_1 ,
\]
since $h^{1/4},h^{2/3}\le h^{1/6}$. A fixed small radius absorbs the last
term. This proves \eqref{eq:sn-ordering} at $\theta=1$. Estimates \eqref{eq:linear-H1} and \eqref{eq:physH1} then yield \eqref{eq:sn-full}. Finally
$\norm{\N(u_h(s))}_H\le C\varepsilon_1^2s_\wedge^{-3/4}e^{-2\nu_1s}$ by
Lemma~\ref{lem:H1stable}. Its convolution with the $H\to H^1$ kernel yields $\norm{z_h(t)}_{H^1}\le C\varepsilon_1^2t_\wedge^{-1/2}e^{-\nu t}$.
The bound for the continuous share $u-v$ is nonsingular, by
\eqref{eq:continuous}. This proves \eqref{eq:sn-horizontal} at
$\theta=1$.\hfill$\square$

\smallskip
For $\theta=2$ put $\varepsilon=\varepsilon_2$.

\begin{lemma}[Time derivatives]\label{lem:time}
On a sufficiently small fixed ball of $V_2$, for $t>0$ and for both the
continuous and the discrete solutions,
\begin{align}
 \norm{\dot u_\bullet(t)}_H+\norm{u_\bullet(t)}_{H^2}
       &\le C\varepsilon t_\wedge^{-1/2}e^{-\nu_1t},\label{eq:q0}\\
 \norm{\dot z_\bullet(t)}_H+\norm{z_\bullet(t)}_{H^2}
       &\le C\varepsilon^2t_\wedge^{-1/2}e^{-\nu_1t},\label{eq:qz0}\\
 \norm{\dot u_\bullet(t)}_{H^1}
       &\le C\varepsilon t_\wedge^{-2}e^{-\nu_1t},\label{eq:q1}\\
 \norm{\dot z_\bullet(t)}_{H^1}
       &\le C\varepsilon^2t_\wedge^{-2}e^{-\nu_1t}.\label{eq:qz1}
\end{align}
For the continuous nonlinear correction,
\begin{equation}\label{eq:zregularity}
 \norm{z(t)}_{H^3}\le C\varepsilon^2t_\wedge^{-1/2}e^{-\nu_1t},
 \qquad
 \norm{z(t)}_{H^4}\le C\varepsilon^2t_\wedge^{-2}e^{-\nu_1t}.
\end{equation}
\end{lemma}
\begin{proof}
For each fixed mesh the differentiated equation has the mild form
\begin{equation}\label{eq:q-mild}
 \dot u_h(t)=A_hT_h(t)\Pi_hu_0+T_h(t)\Pi_h\N(\Pi_hu_0)
 +\int_0^tT_h(t-s)\Pi_h\,2F(u_h(s),\dot u_h(s))\,ds .
\end{equation}
By \eqref{eq:oneway} with $\alpha=2/3$ and \eqref{eq:cp}, the first term is
bounded by $Ct_\wedge^{-1/2}e^{-\lambda t}\varepsilon$. The $V_1$
stability of $\Pi_h$ bounds the second by $C\varepsilon^2e^{-\lambda t}$.

Let $G=\norm{B_hu_h}_H$. By the equation and \eqref{eq:disenergy}, we have
$G\le\norm{\dot u_h}_H+C|u_h|_1^2\le\norm{\dot u_h}_H+C\norm{u_h}_HG$.
The energy bound and a small radius absorb the last term. Hence
$G\le2\norm{\dot u_h}_H$ and
$|u_h|_1\le C\norm{u_h}_H^{1/2}\norm{\dot u_h}_H^{1/2}$. By the weak product estimate,
$\norm{F(u_h,\dot u_h)}_{-1,h}\le C|u_h|_1\norm{\dot u_h}_H$. For
$Q_1=\sup_{0<t\le T_0}e^{\nu_1t}t_\wedge^{1/2}\norm{\dot u_h(t)}_H$,
it follows from \eqref{eq:q-mild} and \eqref{eq:weaksmoothing} that
\[
 Q_1\le C(\varepsilon+\varepsilon^2)+C\varepsilon^{1/2}Q_1^{3/2},
\]
with the same beta integral as in Lemma~\ref{lem:H1stable}. For each mesh
$t^{1/2}\norm{\dot u_h(t)}_H\to0$ as $t\to0$. A continuity argument yields $Q_1\le C\varepsilon$, uniformly in $h$. Graph recovery proves the
discrete part of \eqref{eq:q0}. The bound for $Q_1$ also yields $|u_h|_1\le C\varepsilon t_\wedge^{-1/4}e^{-\nu_1t}$.

Subtracting the derivative of the linear solution from \eqref{eq:q-mild}, we obtain
\[
 \dot z_h=T_h(t)\Pi_h\N(\Pi_hu_0)+
 \int_0^tT_h(t-s)\Pi_h\,2F(u_h,\dot u_h)\,ds .
\]
Hence $\norm{\dot z_h(t)}_H\le C\varepsilon^2t_\wedge^{-1/2}e^{-\nu_1t}$.
Also $\norm{\N(u_h)}_H\le C|u_h|_1^2\le C\norm{u_h}_H\norm{B_hu_h}_H$.
By the equation $A_hz_h=\dot z_h-\Pi_h\N(u_h)$ and graph recovery, we obtain \eqref{eq:qz0} for the discrete solutions.

For \eqref{eq:q1} fix $t>0$ and take $\delta=ct_\wedge^2\le t/4$. On $[t-\delta,t]$, it follows from \eqref{eq:q0} that
$\norm{u_h(s)}_{H^2}\le C\varepsilon t_\wedge^{-1/2}e^{-\nu_1s}$.
By the strong product estimate
$\norm{2F(u_h,\dot u_h)}_H\le C\norm{u_h}_{H^2}
\norm{\dot u_h}_{H^1}$ and the $H\to H^1$ kernel, the forcing term is bounded
in the norm
$\sup_{0<r'\le\delta}r'^{3/4}\norm{\dot u_h(t-\delta+r')}_{H^1}$, which is finite
for each mesh. The coefficient is at most
$C\varepsilon t_\wedge^{-1/2}\delta^{1/4}=C\varepsilon c^{1/4}$. Choose
$c\le1/4$ so small that this coefficient is at most $1/2$ on the fixed ball. At $r'=\delta$,
\[
 \norm{\dot u_h(t)}_{H^1}\le C\delta^{-3/4}\norm{\dot u_h(t-\delta)}_H
 \le C\varepsilon t_\wedge^{-2}e^{-\nu_1t}.
\]
On the same interval the strong forcing in the
$\dot z_h$ equation is bounded by
$C\varepsilon^2t_\wedge^{-5/2}e^{-2\nu_1s}$. The homogeneous term costs
$\delta^{-3/4}$ and the forcing convolution costs $\delta^{1/4}$. With \eqref{eq:qz0}, this yields \eqref{eq:qz1} for the discrete solutions.

All discrete bounds are uniform in $h$. The continuous solutions are
treated in the proof of Lemma~\ref{lem:holder}, by letting $h\to0$.
\end{proof}

The next step controls the time derivative of the ordering error by its
bound in $H$ and a mesh-uniform modulus of continuity.

\begin{lemma}[Time modulus]\label{lem:holder}
For $0<r'\le t_\wedge/8$,
\begin{equation}\label{eq:holder}
 \norm{\dot z_\bullet(t+r')-\dot z_\bullet(t)}_H
 \le C\varepsilon^2e^{-\nu_1t}t_\wedge^{-5/2}r'^{2/3}.
\end{equation}
Consequently
\begin{equation}\label{eq:Ctime}
 \norm{\dot\Cd_h(t)}_H\le Ch^{4/15}\varepsilon^2t_\wedge^{-3/2}e^{-\nu_1t}.
\end{equation}
\end{lemma}
\begin{proof}
The contraction property and \eqref{eq:cp} imply, for $0<r'\le1$,
$\norm{T_h(\sigma+r')-T_h(\sigma)}_{H\to H}\le
C\min\{1,r'\sigma_\wedge^{-3/2}\}e^{-\lambda\sigma}$. Integrating in $\sigma$ and splitting at $\sigma=r'^{2/3}$, we obtain
\begin{equation}\label{eq:increment-kernel}
 \int_0^\infty\norm{T_h(\sigma+r')-T_h(\sigma)}_{H\to H}
 \,d\sigma\le Cr'^{2/3}.
\end{equation}
We restart the $\dot z_h$ equation at $t-t_\wedge/4$. By the discrete
bounds of Lemma~\ref{lem:time}, its strong $H$ forcing on the whole interval up to
$t+r'$ is bounded by $C\varepsilon^2t_\wedge^{-5/2}e^{-\nu_1t}$. The
increment of the homogeneous term is at most
$Cr'\varepsilon^2t_\wedge^{-2}e^{-\nu_1t}$. The forcing integral over the original interval is
bounded by means of \eqref{eq:increment-kernel}. The added interval costs $Cr'$
times the forcing bound. Since $r'\le1$, this proves \eqref{eq:holder}
for the discrete solutions, uniformly in $h$.

\emph{The continuous solutions.} Let $h\to0$. Since $z-z_h=\Cd_h$ and
$u-u_h=e$, Lemma~\ref{lem:prelim} with $s=2$ gives $z_h(t)\to z(t)$ and
$u_h(t)\to u(t)$ in $H$, uniformly in $t\ge0$. Fix $t>0$ and
$0<r'\le t_\wedge/8$. For two meshes $h,\hat h$, put $g=z_h-z_{\hat h}$. Then
$r'\dot g(t)=g(t+r')-g(t)-\int_t^{t+r'}(\dot g(\tau)-\dot g(t))\,d\tau$,
and \eqref{eq:holder} for both discrete solutions gives
\[
 \norm{\dot z_h(t)-\dot z_{\hat h}(t)}_H\le\frac2{r'}\sup_{s\ge0}
 \norm{z_h(s)-z_{\hat h}(s)}_H+C\varepsilon^2t_\wedge^{-5/2}r'^{2/3}.
\]
Taking first $r'$ small and then $h$ and $\hat h$ small, we see that
$\dot z_h(t)$ converges in $H$, uniformly on compact subsets of
$(0,\infty)$. Passing to the limit in
$z_h(t)-z_h(s)=\int_s^t\dot z_h\,d\tau$ shows that $z$ is differentiable for
$t>0$, with $\dot z(t)=\lim\dot z_h(t)$, and that \eqref{eq:holder} holds
for $z$. The bounds \eqref{eq:q0}--\eqref{eq:qz1} for the discrete
solutions are uniform in $h$. Bounded sets of $H^1$ and $H^2$ are weakly
sequentially compact, and the weak limits are identified by the
convergence in $H$. Weak lower semicontinuity of the norms gives the
bounds for $u(t)$, $z(t)$ and $\dot z(t)$ in
\eqref{eq:q0}--\eqref{eq:qz1}. The bounds for $\dot u=\dot v+\dot z$
follow, since $\dot v=AT(\cdot)u_0$ satisfies them by
\eqref{eq:ctsmoothing} and the semigroup property.

For $\psi\in V_2$ the discrete equation
$\dot z_h=A_hz_h+\Pi_h\N(u_h)$ gives
$(\dot z_h-\N(u_h),\Pi_h\psi)_\rho=-b(z_h,\Pi_h\psi)$. Here
$\Pi_h\psi\to\psi$ in $V_1$ by \eqref{eq:approx}, $z_h(t)\to z(t)$ weakly
in $H^2$, and $\N(u_h(t))\to\N(u(t))$ in $H$ by \eqref{eq:prodenergy}, the $H^2$
bounds and interpolation. Hence $(\dot z-\N(u),\psi)_\rho=-b(z,\psi)$ for
$\psi\in V_2$, and by density for $\psi\in V_1$. As in Step~3 of
Section~\ref{sec:convproof}, $z(t)\in D(A)$ and
$Az=\dot z-\N(u)$, that is,
\[
 z'''=-\dot z+\N(u),\qquad
 \norm{\N(u)}_{H^1}\le C\norm u_{H^2}^2\le C\varepsilon^2e^{-2\nu_2t},
\]
where the continuous $V_2$ bound \eqref{eq:cts-strong} is used. With
\eqref{eq:qz0} and \eqref{eq:qz1} this yields \eqref{eq:zregularity},
which completes the proof of Lemma~\ref{lem:time}.

We apply the modulus to $\dot\Cd_h=\dot z-\dot z_h$. By the fundamental theorem of calculus and Lemma~\ref{lem:prelim} with $s=2$, we obtain
\[
 \norm{\dot\Cd_h(t)}_H\le C\varepsilon^2e^{-\nu_1t}
 \Big\{\frac{h^{2/3}}{r'}+t_\wedge^{-5/2}r'^{2/3}\Big\}.
\]
Choose $r'=\frac18h^{2/5}t_\wedge^{3/2}$, which lies in the permitted
range. Both terms are then bounded by $Ch^{4/15}t_\wedge^{-3/2}$. The
exponent $4/15=(2/3)(2/5)$ comes from an $h^{2/3}$ error and a time modulus of exponent
$2/3$.
\end{proof}

Let $R_h$ be the Ritz projection of Lemma~\ref{lem:projection}. By
\eqref{eq:Ritz}, \eqref{eq:zregularity} and
$\min\{a_1,a_2\}\le a_1^{1/3}a_2^{2/3}$,
\begin{equation}\label{eq:Ritz-z}
 \norm{z-R_hz}_{H^2}
 \le C\varepsilon^2e^{-\nu_1t}\min\{t_\wedge^{-1/2},ht_\wedge^{-2}\}
 \le Ch^{2/3}\varepsilon^2t_\wedge^{-3/2}e^{-\nu_1t}.
\end{equation}

\emph{Proof of the estimates at $\theta=2$.} Set $x_h=R_hz-z_h$. Since
$A_hR_hz=\Pi_hAz$, the two equations yield the exact identity
\[
 A_hx_h=\Pi_h\dot\Cd_h-\Pi_h\{\N(u)-\N(u_h)\}.
\]
Graph recovery yields
\begin{equation}\label{eq:recover}
 \norm{\Cd_h}_{H^2}\le\norm{z-R_hz}_{H^2}+C\norm{\dot\Cd_h}_H
 +C\norm{\N(u)-\N(u_h)}_H .
\end{equation}
By the strong product estimate, Sobolev interpolation and Lemma~\ref{lem:time}, we have
$\norm{\N(u)-\N(u_h)}_H\le C\varepsilon t_\wedge^{-1/2}e^{-\nu_1t}
\norm e_H^{1/2}\norm e_{H^2}^{1/2}$. Put
$X_2=e^{\nu_1t}\norm{\Cd_h(t)}_{H^2}$. We use $\norm e_H\le Ch^{2/3}\varepsilon e^{-\nu_1t}$, from
Lemma~\ref{lem:prelim}, and $e=-L_h(\cdot)u_0+\Cd_h$ with
\eqref{eq:linear-H2}. After discarding a decaying factor, we apply Young's inequality to obtain
\begin{align*}
 e^{\nu_1t}C\norm{\N(u)-\N(u_h)}_H
 &\le Ch^{1/2}t_\wedge^{-1}\varepsilon^2
 +Ch^{1/3}t_\wedge^{-1/2}\varepsilon^{3/2}X_2^{1/2}\\
 &\le\tfrac12X_2+Ch^{1/2}t_\wedge^{-1}\varepsilon^2
 +Ch^{2/3}t_\wedge^{-1}\varepsilon^3 .
\end{align*}
We insert \eqref{eq:Ctime} and \eqref{eq:Ritz-z} into \eqref{eq:recover} and absorb $X_2/2$. Since $h\le1$, $t_\wedge\le1$ and $\varepsilon\le r$, we obtain
\[
 \norm{\Cd_h(t)}_{H^2}\le Ch^{4/15}t_\wedge^{-3/2}e^{-\nu_1t}
 \varepsilon^2 .
\]
This is \eqref{eq:sn-ordering} at $\theta=2$. The full estimate follows
from \eqref{eq:splitting}, \eqref{eq:linear-H2} and \eqref{eq:physH2},
and the bound for $u_h-v_h$ from \eqref{eq:qz0}.\hfill$\square$

\subsection{Proof of Theorems~\ref{thm:samenorm} and~\ref{thm:ordering}}
\label{sec:proofsame}
The four solutions exist globally by Theorem~\ref{thm:continuous} and
Proposition~\ref{prop:nonlinstab}. The identity \eqref{eq:splitting} holds by
construction.

\emph{Theorem~\ref{thm:samenorm}.} At $\theta=0$, \eqref{eq:sn-linear}
follows from Proposition~\ref{thm:positive} and the contraction property. At
$\theta=1,2$ it follows from \eqref{eq:linear-H1}, \eqref{eq:linear-H2}
and Proposition~\ref{prop:H1linear} with $s=\theta$. The estimate
\eqref{eq:sn-full} follows from \eqref{eq:splitting},
\eqref{eq:sn-linear} and \eqref{eq:sn-ordering}. Since $a_\theta$ and
$b_\theta$ are nonincreasing in $t$, uniform convergence on $[t_0,\infty)$
follows.

\emph{Theorem~\ref{thm:ordering}.} At $\theta=0$, \eqref{eq:sn-ordering}
including $t=0$ is \eqref{eq:energyfulldefect}. The Taylor tail is
\eqref{eq:energyfulltail}, and the bound for $D_h$ is
Proposition~\ref{prop:energyquadratic}. For $u-v$, the
bound \eqref{eq:sn-horizontal} at each $\theta=0,1,2$ follows from the
continuous Taylor remainder \eqref{eq:taylor} with $m=1$. For $u_h-v_h$, Proposition~\ref{prop:uniformanalytic} yields this bound at
$\theta=0$. The bounds for $u_h-v_h$ and the remaining estimates for the ordering
error at $\theta=1,2$ were proved in Section~\ref{sec:strong}.
At $\theta=1$ we use $V_1=H_0^1$ with equivalent
norms. At $\theta=2$ we use the equivalence of the $V_2$ and $H^2$
norms on $V_2$. The radius was reduced finitely many times, each time to a value depending
only on $L$, $\kappa$ and the fixed rates.\hfill$\square$

\begin{remark}[Support from the companion paper]\label{rem:weight}
The error estimates of this section rest on the mesh-uniform regularity of
Section~\ref{sec:discrete}, and this regularity comes from the weight. The
theoretical and numerical results in the companion paper~\cite{Companion}
support those of this section in three ways.

First, the estimates can fail without this regularity. The standard Galerkin method on
$X_h$ in the unweighted space $L^2(0,L)$, called the unweighted scheme below, loses it. On fine uniform meshes
it has spurious high-frequency modes, which oscillate at a high frequency
of order $h^{-3}$ but are damped only at a rate bounded
independently of $h$. Data built
from these modes make the estimates at $\theta=0$ of
Theorems~\ref{thm:samenorm} and~\ref{thm:ordering} fail for that scheme.
By Theorems~\ref{thm:ctcp} and~\ref{thm:cp}, the continuous problem and
the weighted scheme have no such modes.

Second, the numerical results in~\cite{Companion} agree with the estimates of this paper
for the weighted scheme. The computed value of
$\sup_{0<t\le1}t^{3/2}\norm{A_hT_h(t)}_{H\to H}$ lies between $0.33$ and
$0.39$ on meshes with $64$ to $512$ elements. This agrees with
\eqref{eq:cp} at $\alpha=1$, where the constant does not depend on $h$.
For data of different smoothness, the computed ordering error scales like
$\varepsilon_0^2$, as in Theorem~\ref{thm:ordering}, and decreases under
refinement.
Let $f_h\in X_h$ be a datum built from the spurious modes, with
$\norm{f_h}_2=1$. Take $L=1$, $u_0=f_h/8$ and $t\approx6.3\times10^{-4}$,
on meshes with $512$, $1024$ and $2048$ elements. The computed full error of the
weighted scheme is at most $2.2\times10^{-9}$, and its ordering error
falls from $2.3\times10^{-9}$ to $1.4\times10^{-10}$. For the unweighted
scheme the two errors stay at $4.6\times10^{-2}$ and $9.6\times10^{-6}$.

Third, the companion paper gives an example of an application. It assesses
the coarse-mesh correction of Corollary~\ref{cor:twomesh}, whose
nonlinear part is computed on a coarse mesh. The numerical results show
that the corrected solution is about as accurate as the nonlinear
solution on the fine mesh and takes less than half the computing time.
The estimate of the ordering error gives an a priori error bound for this
use of a coarse mesh.
\end{remark}

\bigskip
{\bf Declarations.} \emph{Use of AI tools.} During the preparation of
this work, the author used Claude Opus~5.5, Claude Fable~5.2 and GPT-6
Astra. Under the author's direction, these tools drafted parts of the text
and of the proofs, checked mathematical arguments, searched for relevant
literature, and assisted with the English and with numerical experiments.
Following a procedure designed by the author, they were also used to carry
out a Lean~4 check of Theorems~\ref{thm:convergence}--\ref{thm:higher}
under certain premises. The author checked and revised all drafts and
assumes responsibility for all content.

\end{document}